\documentclass[10pt]{amsart}
\usepackage{amsfonts}
\usepackage{mathrsfs}
\usepackage{amscd}
\usepackage{amsmath}
\usepackage{amssymb}
\usepackage{latexsym}
\usepackage{lscape}
\usepackage{xypic}
\usepackage{comment}
\usepackage{amscd}
\usepackage{wasysym}  
\usepackage{tikz} 
\usetikzlibrary{matrix,arrows}
\usepackage{appendix}

\newtheorem{thm}{Theorem}[section]

\newtheorem{prop}[thm]{Proposition}
\newtheorem{lem}[thm]{Lemma}

\newtheorem{lem-def}[thm]{Lemma-Definition}
\newtheorem{cor}[thm]{Corollary}

\theoremstyle{definition}
\newtheorem*{ack}{Acknowledgement}

\newtheorem{rmk}[thm]{Remark}
\newtheorem{dfn}[thm]{Definition}

\numberwithin{equation}{section}

\newcommand{\nc}{\newcommand}

\nc{\on}{\operatorname}
\nc{\fraka}{{\mathfrak a}}
\nc{\frakb}{{\mathfrak b}}
\nc{\frakc}{{\mathfrak c}}
\nc{\frakd}{{\mathfrak d}}
\nc{\frake}{{\mathfrak e}}
\nc{\frakf}{{\mathfrak f}}
\nc{\frakg}{{\mathfrak g}}
\nc{\frakh}{{\mathfrak h}}
\nc{\fraki}{{\mathfrak i}}
\nc{\frakj}{{\mathfrak j}}
\nc{\frakk}{{\mathfrak k}}
\nc{\frakl}{{\mathfrak l}}
\nc{\frakm}{{\mathfrak m}}
\nc{\frakn}{{\mathfrak n}}
\nc{\frako}{{\mathfrak o}}
\nc{\frakp}{{\mathfrak p}}
\nc{\frakq}{{\mathfrak q}}
\nc{\frakr}{{\mathfrak r}}
\nc{\fraks}{{\mathfrak s}}
\nc{\frakt}{{\mathfrak t}}
\nc{\fraku}{{\mathfrak u}}
\nc{\frakv}{{\mathfrak v}}
\nc{\frakw}{{\mathfrak w}}
\nc{\frakx}{{\mathfrak x}}
\nc{\fraky}{{\mathfrak y}}
\nc{\frakz}{{\mathfrak z}}
\nc{\frakA}{{\mathfrak A}}
\nc{\frakB}{{\mathfrak B}}
\nc{\frakC}{{\mathfrak C}}
\nc{\frakD}{{\mathfrak D}}
\nc{\frakE}{{\mathfrak E}}
\nc{\frakF}{{\mathfrak F}}
\nc{\frakG}{{\mathfrak G}}
\nc{\frakH}{{\mathfrak H}}
\nc{\frakI}{{\mathfrak I}}
\nc{\frakJ}{{\mathfrak J}}
\nc{\frakK}{{\mathfrak K}}
\nc{\frakL}{{\mathfrak L}}
\nc{\frakM}{{\mathfrak M}}
\nc{\frakN}{{\mathfrak N}}
\nc{\frakO}{{\mathfrak O}}
\nc{\frakP}{{\mathfrak P}}
\nc{\frakQ}{{\mathfrak Q}}
\nc{\frakR}{{\mathfrak R}}
\nc{\frakS}{{\mathfrak S}}
\nc{\frakT}{{\mathfrak T}}
\nc{\frakU}{{\mathfrak U}}
\nc{\frakV}{{\mathfrak V}}
\nc{\frakW}{{\mathfrak W}}
\nc{\frakX}{{\mathfrak X}}
\nc{\frakY}{{\mathfrak Y}}
\nc{\frakZ}{{\mathfrak Z}}
\nc{\bbA}{{\mathbb A}}
\nc{\bbB}{{\mathbb B}}
\nc{\bbC}{{\mathbb C}}
\nc{\bbD}{{\mathbb D}}
\nc{\bbE}{{\mathbb E}}
\nc{\bbF}{{\mathbb F}}
\nc{\bbG}{{\mathbb G}}
\nc{\bbH}{{\mathbb H}}
\nc{\bbI}{{\mathbb I}}
\nc{\bbJ}{{\mathbb J}}
\nc{\bbK}{{\mathbb K}}
\nc{\bbL}{{\mathbb L}}
\nc{\bbM}{{\mathbb M}}
\nc{\bbN}{{\mathbb N}}
\nc{\bbO}{{\mathbb O}}
\nc{\bbP}{{\mathbb P}}
\nc{\bbQ}{{\mathbb Q}}
\nc{\bbR}{{\mathbb R}}
\nc{\bbS}{{\mathbb S}}
\nc{\bbT}{{\mathbb T}}
\nc{\bbU}{{\mathbb U}}
\nc{\bbV}{{\mathbb V}}
\nc{\bbW}{{\mathbb W}}
\nc{\bbX}{{\mathbb X}}
\nc{\bbY}{{\mathbb Y}}
\nc{\bbZ}{{\mathbb Z}}
\nc{\calA}{{\mathcal A}}
\nc{\calB}{{\mathcal B}}
\nc{\calC}{{\mathcal C}}
\nc{\calD}{{\mathcal D}}
\nc{\calE}{{\mathcal E}}
\nc{\calF}{{\mathcal F}}
\nc{\calG}{{\mathcal G}}
\nc{\calH}{{\mathcal H}}
\nc{\calI}{{\mathcal I}}
\nc{\calJ}{{\mathcal J}}
\nc{\calK}{{\mathcal K}}
\nc{\calL}{{\mathcal L}}
\nc{\calM}{{\mathcal M}}
\nc{\calN}{{\mathcal N}}
\nc{\calO}{{\mathcal O}}
\nc{\calP}{{\mathcal P}}
\nc{\calQ}{{\mathcal Q}}
\nc{\calR}{{\mathcal R}}
\nc{\calS}{{\mathcal S}}
\nc{\calT}{{\mathcal T}}
\nc{\calU}{{\mathcal U}}
\nc{\calV}{{\mathcal V}}
\nc{\calW}{{\mathcal W}}
\nc{\calX}{{\mathcal X}}
\nc{\calY}{{\mathcal Y}}
\nc{\calZ}{{\mathcal Z}}

\nc{\scrA}{{\mathscr A}}
\nc{\scrB}{{\mathscr B}}
\nc{\scrR}{{\mathscr R}}

\nc{\bnu}{{\bar{ \nu}}}

\nc{\olO}{\bar{\calO}}

\nc{\al}{{\alpha}} 
\nc{\be}{{\beta}}
\nc{\ga}{{\gamma}} \nc{\Ga}{{\Gamma}}
 \nc{\hGa}{\hat{\Gamma}}
\nc{\ve}{{\varepsilon}} 
\nc{\la}{{\lambda}} \nc{\La}{{\Lambda}}
\nc{\om}{\omega} \nc{\Om}{\Omega} 
\nc{\sig}{{\sigma}} \nc{\Sig}{{\Sigma}}

\nc{\tnb}{\psi_{\rm tame}}

\nc{\op}{{\on{op}}}
\nc{\ad}{{\on{ad}}}
\nc{\alg}{{\on{alg}}}
\nc{\Ad}{{\on{Ad}}}
\nc{\Adm}{{\on{Adm}}} \nc{\aff}{{\on{af}}}
\nc{\Aut}{{\on{Aut}}}
\nc{\Bun}{{\on{Bun}}}
\nc{\cha}{{\on{char}}}
\nc{\der}{{\on{der}}}
\nc{\Der}{{\on{Der}}}
\nc{\diag}{{\on{diag}}}
\nc{\End}{{\on{End}}}
\nc{\Fl}{{\calF\!\ell}}
\nc{\Gal}{{\on{Gal}}}
\nc{\Gr}{{\on{Gr}}}
\nc{\rH}{{\on{H}}}
\nc{\Hom}{{\on{Hom}}}
\nc{\IC}{{\on{IC}}}
\nc{\id}{{\on{id}}}
\nc{\Id}{{\on{Id}}}
\nc{\ind}{{\on{ind}}}
\nc{\Ind}{{\on{Ind}}}
\nc{\Lie}{{\on{Lie}}}
\nc{\Pic}{{\on{Pic}}}
\nc{\pr}{{\on{pr}}}
\nc{\Res}{{\on{Res}}}
\nc{\res}{{\on{res}}} \nc{\Sat}{{\on{Sat}}}
\nc{\s}{{\on{sc}}}
\nc{\drv}{{\on{der}}}
\nc{\sgn}{{\on{sgn}}}
\nc{\Spec}{{\on{Spec}}}\nc{\Spf}{\on{Spf}} 
\nc{\Sph}{\on{Sph}}
\nc{\St}{{\on{St}}}
\nc{\tr}{{\on{tr}}}
\nc{\Tr}{{\on{Tr}}}
\nc{\Mod}{{\mathrm{-Mod}}}
\nc{\Hilb}{{\on{Hilb}}} 
\nc{\Ext}{{\on{Ext}}} 
\nc{\vs}{{\on{Vec}}}
\nc{\ev}{{\on{ev}}}
\nc{\nO}{{\breve{\calO}}}
\nc{\tS}{{\tilde{S}}}
\nc{\spe}{{\on{sp}}}

\nc{\nscrR}{{\mathscr{R}^{\on{nr}}}}

\nc{\GL}{{\on{GL}}}
\nc{\U}{{\on{U}}}
\nc{\Gl}{\on{Gl}} 
\nc{\GSp}{{\on{GSp}}}
\nc{\gl}{{\frakg\frakl}}
\nc{\SL}{{\on{SL}}} 
\nc{\SU}{{\on{SU}}} 
\nc{\SO}{{\on{SO}}}

\nc{\Conv}{{\on{Conv}}}
\nc{\Rep}{{\on{Rep}}}
\nc{\Dom}{{\on{Dom}}}
\nc{\red}{{\on{red}}}
\nc{\act}{{\on{act}}}
\nc{\nr}{{\on{nr}}}

\nc{\str}{{\on{-}}} 
\nc{\os}{{\bar{s}}}
\nc{\oeta}{{\bar{\eta}}}

\nc{\hookto}{\hookrightarrow}
\nc{\longto}{\longrightarrow}
\nc{\leftto}{\leftarrow}
\nc{\onto}{\twoheadrightarrow}
\nc{\lonto}{\twoheadleftarrow}

\nc{\bio}{{\bar{i}}}
\nc{\bjay}{{\bar{j}}}

\nc{\bFl}{{\overline{\Fl}}} 
\nc{\bU}{{\overline{U}}}
\nc{\tGr}{{\tilde{\Gr}}}
\nc{\cGr}{\calG\! r}
\nc{\oGr}{\overline{\on{Gr}}} 
\nc{\ocGr}{\overline{\calG\! r}}

\nc{\ohtimes}{\stackrel{!}{\otimes}}
\nc{\boxtilde}{\widetilde{\boxtimes}}
\nc{\vstar}{{\varhexstar}}

\nc{\bslash}{\backslash}
\nc{\algQl}{{\bar{\bbQ}_\ell}}
\nc{\sF}{{\bar{F}}}
\nc{\nF}{{\breve{F}}}
\nc{\nW}{{W^{\on{nr}}}}
\nc{\sk}{{\bar{k}}}
\nc{\cont}{\on{c}}
\nc{\supp}{\on{supp}}
\nc{\blt}{\bullet}  
\nc{\dom}{\on{dom}}
\nc{\scon}{{\on{sc}}} 
\nc{\Affine}{\on{Aff}} 
\nc{\nscrA}{\mathscr{A}^{\on{nr}}} 
\nc{\nfraka}{{\fraka^{\on{nr}}}}
\nc{\ran}{{\rangle}}
\nc{\lan}{{\langle}}
\nc{\bk}{{\bar{k}}}
\nc{\tF}{{\tilde{F}}}
\nc{\LG}{{^\text{L}\hspace{-0.04cm}G}}
\nc{\LL}{{^\text{L}\hspace{-0.07cm}L}}

\nc{\pot}[1]{ [\hspace{-0,5mm}[ {#1} ]\hspace{-0,5mm}] }
\nc{\rpot}[1]{ (\hspace{-0,7mm}( {#1} )\hspace{-0,7mm}) }

\nc{\defined}{\hspace{0.1cm}\stackrel{\text{\rm \tiny def}}{=}\hspace{0.1cm}}

\begin{document}

\title[Geometric Satake]{Affine Grassmannians and geometric Satake equivalences}
\author[T. Richarz]{by Timo Richarz}

\address{Timo Richarz: Mathematisches Institut der Universit\"at Bonn, Endenicher Allee 60, 53115 Bonn, Germany}
\email{richarz@math.uni-bonn.de}

\maketitle

\begin{abstract} 
I extend the ramified geometric Satake equivalence of Zhu \cite{RZ} from tamely ramified groups to include the case of general connected reductive groups. As a prerequisite I prove basic results on the geometry of affine flag varieties. 
\end{abstract}

\tableofcontents
\thispagestyle{empty}

\section*{Introduction}
Let $k$ be an algebraically closed field. Let $G$  be a connected reductive group over the Laurent power series local field $F=k\rpot{t}$. The \emph{(twisted) loop group} $LG$ is the functor on the category of $k$-algebras
\[LG\colon R\;\longmapsto\;G(R\rpot{t}).\]
The loop group is representable by a strict ind-affine ind-group scheme over $k$, cf. Pappas-Rapoport \cite{PR}. Let $\calG$ be a smooth affine model of $G$ over $\calO_F=k\pot{t}$, i.e. a smooth affine group scheme over $\calO_F$ with generic fiber $G$. The \emph{(twisted) positive loop group} $L^+\calG$ is the functor on the category of $k$-algebras
\[L^+\calG\colon R\;\longmapsto\; \calG(R\pot{t}).\]
The positive loop group $L^+\calG$ is representable by a reduced affine subgroup scheme of $LG$ of infinite type over $k$. In general, the loop group $LG$ is neither reduced nor connected, whereas the positive loop group $L^+\calG$ is connected if the special fiber of $\calG$ is connected. 

Our first main result is a basic structure theorem.

\bigskip
\noindent {\bf Theorem A.}
\emph{A smooth affine model of $G$ with geometrically connected fibers $\calG$ over $\calO_F$ is parahoric in the sense of Bruhat-Tits \cite{BT2} if and only if the fpqc-quotient $LG/L^+\calG$ is representable by an ind-proper ind-scheme. In this case, $LG/L^+\calG$ is ind-projective.}
\bigskip

Theorem A should be viewed as the analogue of the characterization of parabolic subgroups in linear algebraic groups by the properness of their fppf-quotient. Note that the proof of the ind-projectivity of $LG/L^+\calG$ for parahoric $\calG$ is implicitly contained in Pappas-Rapoport \cite{PR}.

Let $\scrB(G,F)$ be the extended Bruhat-Tits building. Let $\fraka\subset \scrB(G,F)$ be a facet, and let $\calG_\fraka$ be the corresponding parahoric group scheme. The fpqc-quotient $\Fl_\fraka=LG/L^+\calG_\fraka$ is called the affine flag variety associated with $\fraka$, cf. \cite{PR}. The positive loop group $L^+\calG_\fraka$ acts from the left on $\Fl_\fraka$, and the action on each orbit factors through a smooth affine quotient of $L^+\calG_\fraka$ of finite type. This allows us to consider the category $P_{L^+\calG_\fraka}(\Fl_\fraka)$ of $L^+\calG_\fraka$-equivariant $\ell$-adic perverse sheaves on $\Fl_\fraka$. Here $\ell$ is a prime number different from the characteristic of the ground field $k$. Recall that a facet $\fraka\subset \scrB(G,F)$ is called \emph{special} if it is contained in some apartment such that each wall is parallel to a wall passing through $\fraka$. 

Our second main theorem characterizes special facets $\fraka$ in terms of the category $P_{L^+\calG_\fraka}(\Fl_\fraka)$.

\bigskip
\noindent {\bf Theorem B.} \emph{The following properties are equivalent:\smallskip\\
i) The facet $\fraka$ is special.\smallskip\\
ii) The stratification of $\Fl_\fraka$ in $L^+\calG_\fraka$-orbits satisfies the parity property, i.e. in each connected component all strata are either even or odd dimensional.\smallskip\\
iii) The category $P_{L^+\calG_\fraka}(\Fl_\fraka)$ is semi-simple.}
\bigskip

The implications $i)\Rightarrow ii)\Rightarrow iii)$ are due to Zhu \cite{RZ} whereas the implication $iii)\Rightarrow i)$ seems to be new. In fact, the following properties are equivalent to Theorem B $i)$-$iii)$ (cf. \S 3 below): \smallskip\\
\emph{iv) The special fiber of each global Schubert variety associated with $\fraka$ is irreducible. \smallskip\\
v) Each admissible set associated with $\fraka$ contains a unique maximal element. }\smallskip\\
See \S \ref{globSchub} for the definition of global Schubert varieties and admissible sets associated with a facet. Moreover, if $k$ is the algebraic closure of a finite field, each item $i)$-$v)$ is equivalent to:\smallskip\\
\emph{vi) The monodromy on Gaitsgory's nearby cycles functor associated with $\fraka$ vanishes.\smallskip\\}
See \S \ref{monpara} for the definition of Gaitsgory's nearby cycles functor in this context.
\medskip

Now let $k$ be the algebraic closure of a finite field and let $\fraka$ be a special facet. The \emph{ramified Satake category $\Sat_\fraka$ associated with $\fraka$} is the category
\[\Sat_\fraka\defined P_{L^+\calG_\fraka}(\Fl_\fraka).\]  
The ramified Satake category $\Sat_\fraka$ is semi-simple with simple objects as follows. Let $A$ be a maximal $F$-split torus such that $\fraka\subset \scrA(G,A,F)$ lies in the corresponding apartment. Since $k$ is algebraically closed, $G$ is quasi-split by Steinberg's Theorem. The centralizer $T=Z_G(A)$ is a maximal torus. Let $B$ be a Borel subgroup containing $T$. Let $I=\Gal(\sF/F)$ denote the absolute Galois group. The group $I$ acts on the cocharacter group $X_*(T)$, and we let $X_*(T)_I$ be the group of coinvariants. To every $\mu\in X_*(T)_I$, the Kottwitz morphism associates a $k$-point $t^\mu\cdot e_0$ in $\Fl_\fraka$, where $e_0$ denotes the base point. Let $Y_\mu$ be the reduced $L^+\calG$-orbit closure of $t^\mu\cdot e_0$. The scheme $Y_\mu$ is a projective variety over $k$ which is in general not smooth. The reduced locus of $\Fl_\fraka$ has an ind-presentation
\[(\Fl_\fraka)_\red\;=\;\varinjlim_{\mu\in X_*(T)_I^+}Y_\mu,\]
where $X_*(T)_I^+$ is the image of the set of dominant cocharacters under the canonical projection $X_*(T)\to X_*(T)_I$. Then the simple objects of $\Sat_\fraka$ are the intersection complexes $\IC_\mu$ of $Y_\mu$, as $\mu$ ranges over $X_*(T)_I^+$. 

Recall that for every $\calA_1,\calA_2\in \Sat_\fraka$, the \emph{convolution product} $\calA_1\star\calA_2$ is defined as an object in the bounded derived category of constructible $\ell$-adic complexes, cf. Gaitsgory \cite{Ga}, Pappas-Zhu \cite{PZ}. 

Fix a pinning of $G$ preserved by $I$, and denote by $\hat{G}$ the Langlands dual group over $\algQl$, i.e. the connected reductive group over $\algQl$ whose root datum is dual to the root datum of $G$. The Galois group $I$ acts on $\hat{G}$ via outer automorphisms, and we let $\hat{G}^I$ be the fixed points. Then $\hat{G}^I$ is a not necessarily connected reductive group over $\algQl$. Note that $X_*(T)_I=X^*(\hat{T}^I)$, and that for every $\mu\in X^*(\hat{T}^I)^+$, there exists a unique irreducible representation of $\hat{G}^I$ of highest weight $\mu$, cf. Appendix \ref{fixpointapp} for the definition of highest weight representations of $\hat{G}^I$. Let $\Rep_\algQl(\hat{G}^I)$ be the category of algebraic representations of $\hat{G}^I$. 

Our third main result describes $\Sat_\fraka$ as a tensor category.

\bigskip
\noindent {\bf Theorem C.} \emph{i) The category $\Sat_\fraka$ is stable under the convolution product $\star$, and the pair $(\Sat_\fraka, \star)$ admits a unique structure of a symmetric monoidal category such that the global cohomology functor
\[\om(\str)\defined \bigoplus_{i\in \bbZ}R^i\Ga(\Fl_\fraka,\str)\colon \Sat_\fraka\;\longto\;\vs_{\algQl}\]
is symmetric monoidal.\smallskip\\
ii) The functor $\om$ is a faithful exact tensor functor, and induces via the Tannakian formalism an equivalence of tensor categories
\begin{align*}
(\Sat_\fraka,\star)\;&\overset{\simeq}{\longto}\;(\Rep_\algQl(\hat{G}^I),\otimes),\\
\calA&\longmapsto\;\om(\calA)
\end{align*}
which is uniquely determined up to inner automorphisms of $\hat{G}^{I}$ by elements in $\hat{T}^{I}$ by the property that $\om(\IC_\mu)$ is the irreducible representation of highest weight $\mu$. }
\bigskip

We also prove a variant of Theorem C which includes Galois actions, and where $k$ may be replaced by a finite field and where $G$ is assumed to be quasi-split over $F$, cf. Theorem \ref{Zhu1} below. 

Theorem C is due to Zhu \cite{RZ} in the case of tamely ramified groups. With Theorem B at hand, our method follows the method of \cite{RZ} with minor modifications. The proof uses the unramified Satake equivalence as a black box which we will now recall briefly. 

The affine Grassmannian $\Gr_G$ is the fpqc-sheaf associated with the functor on the category of $F$-algebras $\Gr_G\colon R\mapsto G(R\rpot{z})/G(R\pot{z})$ for an additional formal variable $z$. Denote by $L^+_zG\colon R\mapsto G(R\pot{z})$ the positive loop group formed with respect to $z$. Then $L^+_zG$ acts on $\Gr_G$ from the left. Fix $\sF$ the completion of a separable closure of $F$. The unramified Satake category $\Sat_{G,\sF}$ is the category 
\[\Sat_{G,\sF}\defined P_{L^+_zG_\sF}(\Gr_{G,\sF}),\]
cf. \cite{Ri2}. The category $\Sat_{G,\sF}$ is equipped with the structure of a neutralized Tannakian category with respect to the convolution product $\star$. The \emph{unramified Satake equivalence} is an equivalence of abelian tensor categories 
\[(\Sat_{G,\sF},\star)\;\simeq\;(\Rep_{\algQl}(\hat{G}),\star),\] 
which is uniquely determined up to inner automorphism by elements in $\hat{T}$, cf. Mirkovi\'c-Vilonen \cite{MV}, Richarz \cite{Ri2}.

The main ingredient in the proof of Theorem C is the \emph{BD-Grassmannian $\Gr_\fraka$ associated with $\fraka$} (BD = Beilinson-Drinfeld) which is a strict ind-projective ind-scheme over $\calO_F$ such that in the generic (resp. special) fiber 
\[\Gr_{\fraka,\eta}\;\simeq\;\Gr_G \hspace{0.6cm}\text{(resp. $\Gr_{\fraka,s}\;\simeq\;\Fl_\fraka$)}.\]
This allows us to consider Gaitsgory's nearby cycles functor $\Psi_\fraka\colon \Sat_{G,\sF}\to \Sat_\fraka$ associated with $\Gr_\fraka\to\Spec(\calO_F)$. The symmetric monoidal structure with respect to $\star$ on the category $\Sat_{G,\sF}$ in the generic fiber of $\Gr_\fraka$ extends to the category $\Sat_\fraka$ in the special fiber of $\Gr_\fraka$. This equips $(\Sat_\fraka,\star)$ with a symmetric monoidal structure. Here, the key fact is the vanishing of the monodromy of $\Psi_\fraka$ for special facets $\fraka$, cf. item $vi)$ in the list below Theorem B. It is then not difficult to exhibit $(\Sat_\fraka,\star)$ as a Tannakian category with fiber functor $\om$. Theorem B $iii)$ implies that the neutral component $\Aut^\star(\om)^0$ of the $\algQl$-group of tensor automorphisms is reductive. In fact, the nearby cycles construction above realizes $\Aut^\star(\om)$ as a subgroup of $\hat{G}$ via the unramified Satake equivalence. This equivalence equips $\hat{G}$ with a canonical pinning, and it is easy to identify $\Aut^\star(\om)=\hat{G}^I$ as the subgroup of $\hat{G}$ where $I$ acts by pinning preserving automorphisms.  

In \cite{HRo}, Haines and Rostami establish the Satake isomorphism for Hecke algebras of special parahoric subgroups. Theorem C may be seen as a geometrization of this isomorphism in the case that the facet $\fraka$ is \emph{very special}. Note that in Theorem C above the residue field $k$ was assumed to be algebraically closed, and hence the notion of special facets and very special facets coincide (cf. Definition \ref{veryspecialdfn} below). However, some additional input is needed to identify Theorem C as a geometrization of the isomorphism given in \cite{HRo} in this case. The case of quasi-split connected reductive groups and special facets which are not neccessarily very special will be adressed on another occasion \cite{Ri3}. \smallskip\\  

Let us briefly explain the structure of the paper. \S 1 is devoted to the proof of Theorem A. In \S 2, we introduce the global affine Grassmannian associated with a facet, and define the global Schubert varieties. In \S 3, we prove Theorem B. In \S 4, we collect some facts from the unramified and ramified geometric Satake equivalences, and explain the proof of Theorem C including the case of wild ramification. Appendix \ref{levilemapp} supplements \S 1 and concerns the construction of Bruhat-Tits group schemes associated with Levi subgroups which we need in \S 1.5 for the proof of the ind-projectivity. In Appendix \ref{reductivedescentapp} we give a descent result on reductive groups explained to me by Brian Conrad which we use in \S 2 to extend a parahoric group scheme over a discrete valuation ring to some smooth affine curve. Finally, Appendix \ref{fixpointapp} complements \S 4. Here we introduce highest weight representations for the group of fixed points in a split connected reductive group under pinning preserving automorphisms.

\begin{ack}
This paper is part of the author's doctoral thesis. First of all I thank my advisor M. Rapoport for his steady encouragement and advice during the process of writing. I thank B. Conrad for explaining to me the beautiful proof of the descent result in Appendix \ref{reductivedescentapp}. Furthermore, I thank U. G\"ortz, T. J. Haines, E. Hellmann, B. Levin and P. Scholze for useful dicussions around the subject. I thank the referees for their carful reading of the manuscript and their useful comments which substantially improved the readability. I am grateful to the stimulating working atmosphere in Bonn and for the funding by the Max-Planck society.
\end{ack}

\noindent {\bf Notation.} For a complete discretely valued field $F$, we denote by $\calO_F$ the ring of integers with residue field $k$. We let $\sF$ be the completion of a fixed separable closure, and $\Ga=\Gal(\sF/F)$ the absolute Galois group. The completion of the maximal unramified subextension of $F$ is denoted ${\nF}$ with ring of integers $\calO_{\nF}$ and residue field $\bar{k}$.

\section{Affine Grassmannians}

In \S 1.1 and 1.2, we collect some facts on affine Grassmannians from the literature, cf. \cite{He}, \cite{Levin}, \cite{PR}, \cite{PZ}, \cite{Z}. In \S 1.3-1.5, we prove Theorem A from the introduction.

\subsection{Affine flag varieties}\label{affflagpara} Let $k$ be either a finite or an algebraically closed field, and let $G$ be a connected reductive group over the Laurent series local field $F=k\rpot{t}$. The \emph{(twisted) loop group $LG$} is the group functor on the category of $k$-algebras
\[LG\colon R\;\longto\; G(R\rpot{t}).\]
The loop group $LG$ is representable by a strict ind-affine ind-group scheme, cf. \cite[\S 1]{PR}. Let $\fraka$ be a facet in the enlarged Bruhat-Tits building $\scrB(G,F)$. Denote by $\calG_\fraka$ the associated parahoric group scheme over $\calO_F$, i.e. the neutral component of the unique smooth affine group scheme over $\calO_F$ such that the generic fiber is $G$, and such that the $\calO_F$-points are the pointwise fixer of $\fraka$ in $G(F)$. The \emph{(twisted) positive loop group $L^+\calG_\fraka$} is the group functor on the category of $k$-algebras 
\[L^+\calG_\fraka\colon  R\;\longto\; \calG_\fraka(R\pot{t}).\] 
The positive loop group $L^+\calG_\fraka$ is representable by a reduced affine connected group scheme of infinite type over $k$. Then $L^+\calG_\fraka\subset LG$ is a closed subgroup scheme. The \emph{(partial) affine flag variety $\Fl_\fraka$} is the fpqc-sheaf on the category of affine $k$-algebras associated with the functor
\[\Fl_\fraka\colon  R\;\longto\; LG(R)/L^+\calG_\fraka(R).\]
The affine flag variety $\Fl_\fraka$ is a strict ind-scheme of ind-finite type and separated over $k$, cf. \cite[Theorem 1.4]{PR}. We explain some of its basic structure. For the rest of this subsection, let $k$ be an algebraically closed field. \medskip\\
\emph{Connected components of $\Fl_\fraka$:} In general, $\Fl_\fraka$ is not connected. The connected components are given as follows, cf. \cite{PR}. 

Let $\pi_1(G)$ be the algebraic fundamental group of $G$, cf. Borovoi \cite{Bo}. The group $\pi_1(G)$ is a finitely generated abelian group, and can be defined as the quotient of the coweight lattice by the coroot lattice of $G_\sF$. Since $k$ is algebraically closed, $I=\Gal(\sF/F)$ is the inertia group. Then $I$ acts on $\pi_1(G)$, and we denote by $\pi_1(G)_I$ the group of coinvariants. There is a functorial surjective group morphism $\kappa_G\colon  LG(k)\to \pi_1(G)_I$, cf. Kottwitz \cite[\S 7]{Kott}. By \cite[\S 2.a.2]{PR}, the morphism $\kappa_G$ extends to a locally constant morphism of ind-group schemes
\[\kappa_G\colon  LG\;\longto\; \underline{\pi_1(G)}_I,\]
and induces an isomorphism on the groups of connected components $\pi_0(LG)\simeq \pi_1(G)_I$. Since $L^+\calG_\fraka$ is connected, the morphism $\kappa_G$ also defines a bijection $\pi_0(\Fl_\fraka)\simeq \pi_1(G)_I$.\medskip\\
\emph{Schubert varieties in $\Fl_\fraka$:} The reduced $L^+\calG_\fraka$-orbits in $\Fl_\fraka$ are called Schubert varieties. Their basic structure is as follows. 

Let $A\subset G$ be a maximal $F$-split torus such that $\fraka$ is contained in the apartment $\scrA=\scrA(G,A,F)$ of $\scrB(G,F)$. Let $N=N_G(A)$ be the normalizer of $A$, and denote by $W=N(F)/T_1$ the Iwahori-Weyl group where $T_1\subset T(F)$ is the unique parahoric subgroup, cf. \cite{HR}.  

\begin{dfn}\label{schubertdfn}
For $w\in W$, the \emph{Schubert variety $Y_w$ associated with $w$} is the reduced $L^+\calG_\fraka$-orbit closure 
\[Y_w\defined \overline{L^+\calG_\fraka \cdot n_w\cdot e_0} \;\subset\;\Fl_\fraka,\]
where $n_w$ is a representative of $w$ in $LG(k)$, and $e_0$ is the base point in $\Fl_\fraka$. 
\end{dfn}

Let us justify the definition. The orbit map $L^+\calG_\fraka\to \Fl_\fraka$, $g\mapsto g\cdot n_we_0$ factors through some smooth affine quotient of $L^+\calG_\fraka$ of finite type. The Schubert variety $Y_w\subset \Fl_\fraka$ is the scheme theoretic closure of this morphism, and hence a reduced separated scheme of finite type over $k$. Let $\mathring{Y}_w$ denote the $L^+\calG_\fraka$-orbit of $n_w\cdot e_0$ in $Y_w$. Then $\mathring{Y}_w$ is a smooth open dense subscheme of $Y_w$. Since $L^+\calG_\fraka$ is connected, $\mathring{Y}_w$ is irreducible and so is $Y_w$. It follows that $Y_w$ is a variety over $k$.

The Iwahori-Weyl group $W$ acts on the apartment $\scrA$ by affine transformations. Let $\fraka_C$ be an alcove containing $\fraka$ in its closure. The choice of $\fraka_C$ equips $W$ with a quasi-Coxeter structure $(l,\leq)$, i.e. the simple reflections are the reflections at the walls meeting the closure of $\fraka_C$. Let $W_\fraka=N(F)\cap\calG_\fraka(\calO_F)/T_1$ the subgroup of $W$ associated with $\fraka$, cf. \cite{HR}. The group $W_\fraka$ identifies with the subgroup generated by the reflections at the walls passing through $\fraka$. For an element $w\in W$, denote by $w^\fraka$ the representative of minimal length in $w\cdot W_\fraka$. For every $w\in W$, there is a unique representative ${_\fraka w}^\fraka$ of maximal length in the set
\[\{(w'ww'')^\fraka\;|\;w',w''\in W_\fraka\},\]
cf.  \cite[Lemma 2.15]{Ri1}. Let ${_\fraka W}^\fraka\subset W$ be the subset of all $w\in W$ such that $w={_\fraka w}^\fraka$. Then the canonical map ${_\fraka W}^\fraka\to W_\fraka\bslash W/ W_\fraka$ is bijective. The Bruhat decomposition, cf. \cite{HR}, implies that there is a set-theoretically disjoint union into locally closed strata,
\begin{equation}\label{stratadecom}
\Fl_\fraka \;=\; \coprod_{w\in {_\fraka W}^\fraka}\mathring{Y}_w.
\end{equation}

\begin{prop} \label{schubertstr} Let $w\in {_\fraka W}^\fraka$.\smallskip\\
i) The scheme $\mathring{Y}_{w}$ is of dimension $l(w)$. \smallskip\\
ii) The Schubert variety $Y_w$ is a proper $k$-variety, and
\[Y_w\;=\; \coprod_{v\leq w}\mathring{Y}_v,\]
where $v\in {_\fraka W}^\fraka$ and $\leq$ denotes the Bruhat order on $W$.  
\end{prop}
\begin{proof}
Part i), and the orbit stratification in part ii) is proven in \cite[Proposition 2.8]{Ri1}. Let us show that $Y_w$ is proper. Note that the Iwahori $L^+\calG_{\fraka_C}$ is a closed subgroup scheme of $L^+\calG_\fraka$. After multiplication with some $\tau\in LG(k)$ contained in the stabilizer of $L^+\calG_{\fraka_C}$, we may assume that $Y_w$ is the reduced $L^+\calG_{\fraka_C}$-orbit closure of $n_w\cdot e_0$ contained in the neutral component $(\Fl_\fraka)^0$. Let $\tilde{w}$ be a reduced decomposition of $w$ as a product of simple reflections, and let $\pi_{\tilde{w}}\colon D_{\tilde{w}}\to Y_w$ be the associated Demazure resolution, \cite[\S 8]{PR}. In [\emph{loc. cit.}] full flag varieties are considered, but the composition $D_{\tilde{w}}\to \Fl_{\fraka_C}\to \Fl_\fraka$ factors through $Y_w$ defining $\pi_{\tilde{w}}$. By \cite[Proposition 8.8]{PR}, the scheme $D_{\tilde{w}}$ is projective, hence proper. Because $\pi_{\tilde{w}}$ is surjective and $Y_w$ separated, it follows that $Y_w$ is proper.
\end{proof}

\begin{cor}\label{flagindproper} i) The strict ind-scheme $\Fl_\fraka$ is ind-proper. The reduced locus admits an ind-presentation by Schubert varieties
\[(\Fl_\fraka)_\red\;=\;\varinjlim_{w}Y_w.\]
ii) The stabilizers of the $L^+\calG_\fraka$-action on $\Fl_\fraka$ are connected.
\end{cor}
\begin{proof}
Properness can be checked on the underlying reduced subscheme. Thus, part i) follows from \eqref{stratadecom}, and Proposition \ref{schubertstr}. For ii), it is enough to prove that for $w\in W$, the $L^+\calG_\fraka$-stabilizer in the point $n_w\cdot e_0$ is connected where $n_w\in LG(k)$ denotes a representative of $w$. The stabilizer $L^+\calG_{\fraka,w}$ is the closed subgroup scheme given by
\[L^+\calG_{\fraka,w}=L^+\calG_{\fraka}\cap (n_w\cdot L^+\calG_\fraka\cdot n_w^{-1}).\] 
We have $n_w\cdot L^+\calG_\fraka\cdot n_w^{-1}=L^+\calG_{w \fraka}$ where $w \fraka$ is the $w$-translate of $\fraka$ in the apartment $\scrA$. For a subset $\Omega\subset \scrA$ whose projection onto the semi-simple part $\scrA_{\rm ss}$ is bounded, Bruhat and Tits associate the unique smooth affine $\calO_F$-group scheme $\calG_\Omega$ with generic fiber $G$ and connected special fiber such that the $\calO_F$-points are the pointwise fixer of $\Omega$. This is in accordance with our previous notation, i.e. if $\Omega$ is a facet, then $\calG_\Omega$ is the parahoric group scheme as above. We have $L^+\calG_{\fraka,w} = L^+\calG_{\fraka}\cap L^+\calG_{w\fraka}=L^+\calG_{\fraka\cup w\fraka}$ by standard properties of Bruhat-Tits groups (cf. \cite[7.1.11]{BT1}) and the latter is connected, e.g. Lemma \ref{groupscheme} ii) below. This proves ii). 
\end{proof}

\begin{rmk}
i) Note that $\calG_{\fraka\cup w\fraka}=\calG_\Omega$ where $\Omega$ denotes the convex hull of $\fraka\cup w\fraka$ in $\scrA$, cf. \cite[7.1.9]{BT1}. If $\fraka$ is hyperspecial, i.e. $\calG_\fraka$ is connected reductive, then Corollary \ref{flagindproper} ii) specializes to \cite[Lemme 2.3]{NP}. \smallskip\\
ii) If $\cha(k)=0$, the connectedness of the stabilizers follows from the simply connectedness of the $L^+\calG_\fraka$-orbits as follows: we have a finite \'etale cover
\[L^+\calG_\fraka/L^+\calG_{\fraka,w}^0\;\longto\;  L^+\calG_\fraka/L^+\calG_{\fraka,w} \;\simeq\;\mathring{Y}_w,\]
where $L^+\calG_{\fraka,w}^0$ denotes the neutral component. Since $L^+\calG_\fraka$ is connected and $\mathring{Y}_w$ is simply connected, the cover is trivial, i.e. $L^+\calG_{\fraka,w}=L^+\calG_{\fraka,w}^0$. The orbit $\mathring{Y}_w$ is the extension of a homogenous space by an affine bundle which is, in positive characteristic, not simply connected for the \'etale topology in general, e.g. if $\calG_\fraka$ is an Iwahori group scheme and $w$ is a simple reflection, then $\mathring{Y}_w=\bbA^1_k$ whose fundamental group is huge for $\cha(k)>0$.
\end{rmk}

\subsection{BD-Grassmannians over curves} Let $S$ be a scheme, and let $X$ be a separated $S$-scheme. Let $\calG$ be a fpqc-sheaf of groups over $X$. For a $S$-scheme $T$, denote by $X_T=X\times_ST$ the fiber product. If $x\colon  T\to X$ is a morphism of $S$-schemes, let $\Ga_x\subset X_T$ be the graph.

\begin{dfn} The \emph{BD-Grassmannian $\Gr(\calG,X)$} is the contravariant functor on the category of $S$-schemes parametrizing isomorphism classes of triples $(x,\calF,\al)$ with
\[
\begin{cases}
x\colon T\to X\; \text{is a morphism of $S$-schemes};\\
\calF \;\text{a right $\calG_T$-torsor on $X_T$};\\
\al\colon \calF_{X_T\bslash\Ga_x}\overset{\simeq}{\longto}\calF^0|_{X_T\bslash\Ga_x}\;\text{a trivialization},
\end{cases}
\]
where $\calF^0$ denotes the trivial torsor. Two triples $(x,\calF,\al)$ and $(x',\calF',\al')$ are isomorphic if $x=x'$, and as torsors $\calF\simeq \calF'$ compatible with the trivializations $\al$ and $\al'$.
\end{dfn}

If $\calG=G\times_k X$ is constant, then we write $\Gr(G,X)=\Gr(\calG,X)$. 

\begin{rmk}\label{Heinrmk}
For constant groups the BD-Grassmannian (=Beilinson-Drinfeld) is defined by Beilinson and Drinfeld in \cite{BD}. Note that in general the BD-Grassmannian $\Gr(\calG,X)$ is a special case of Heinloth's construction \cite[\S 2 Example (2)]{He}.
\end{rmk}

There is the canonical projection $\Gr(\calG,X)\to X$, $(x,\calF,\al)\mapsto x$. The construction is functorial in the following sense. If $\tau\colon  \calG'\to\calG$ is a morphism of fpqc-sheaves of groups over $X$, then there is a morphism of functors
\begin{align}\label{fungroup}
\begin{split}
 \tau_*\colon \Gr(\calG',X)&\longto\Gr(\calG,X)\\
(x,\calF,\al)&\longmapsto (x,\tau_*\calF,\tau_*\al).
\end{split}
\end{align} 
If $f\colon  Y\to X$ is a morphism of $S$-schemes, then there is a morphism of functors
\begin{align}\label{funbase}
\begin{split}
f^*\colon \Gr(\calG,X)\times_XY&\longto\Gr(\calG_Y,Y)\\
((x,\calF,\al),y)&\longmapsto(y,\calF_Y,\al_Y).
\end{split}\end{align}

For the rest of the subsection, let $S=\Spec(k)$ be the spectrum of a field $k$, and let $X$ be a smooth connected curve over $k$. 

\begin{lem}\label{repGln}
Let $\calE$ be a vector bundle on $X$. Then $\Gr(\GL(\calE),X)\to X$ is representable by an ind-proper strict ind-scheme which is, Zariski locally on $X$, ind-projective.
\end{lem}
\begin{proof}
This follows from \cite[Lemma 2.4]{Ri2}. 
\end{proof} 

\begin{lem}\label{affgrass1} Let $\calG$ be an affine group scheme of finite type over $X$.\smallskip\\
i) If $\iota\colon  \calP\hookto \calG$ is a closed immersion of group schemes such that the fppf-quotient $\calG/\iota(\calP)$ is affine (resp. quasi-affine) over $X$, then the morphism $\iota_*\colon  \Gr(\calP,X)\to\Gr(\calG,X)$ is relatively representable by a closed (resp. quasi-compact) immersion.\smallskip\\
ii) If $\calG$ is flat over $X$, then $\Gr(\calG,X)$ is strict ind-representable of ind-finite type and separated over $X$.
\end{lem}
\begin{proof} 
Part i) is analogous to the proof of Levin \cite[Theorem 3.3.7]{Levin}. For part ii) choose any vector bundle $\calE$ on $X$ together with a faithful representation $\rho\colon\calG\hookto\GL(\calE)$ such that the fppf-quotient $\GL(\calE)/\rho(\calG)$ is representable by a quasi-affine scheme. The existence of the pair $(\calE,\rho)$ is proven in \cite[\S 2 Example (1)]{He}. Now ii) follows from i) applied to $\rho\colon\calG\hookto\GL(\calE)$ and Lemma \ref{repGln}. 

\end{proof}

If $T=\Spec(R)$ is affine and $x\in X(T)$, let $\hGa_x=\Spec(\hat{\calO}_{X_T,x})$ be the spectrum of the completion of $X_T$ along $\Ga_x$, i.e. if we choose a local parameter $t_x$  at $x$, which exists Zariski locally on $T$, then $\hat{\calO}_{X_T,x}\simeq R\pot{t_x}$. The graph $\Ga_x\hookto\hGa_x$ is a closed subscheme, and we let $\hGa_x^o=\hGa_x\bslash\Ga_x$ be its open complement.  

\begin{dfn}
i) The \emph{global loop group $\calL\calG$} is the functor on the category of $k$-algebras 
\[\calL\calG\colon  R\;\longmapsto\;\{(x,g)\;|\;x\in X(R),\; g\in\calG(\hGa_x^o)\}.\]
ii) The \emph{global positive loop group $\calL^+\calG$} is the functor on the category of $k$-algebras
\[\calL^+\calG\colon  R\;\longmapsto\;\{(x,g)\;|\;x\in X(R),\; g\in\calG(\hGa_x)\}.\]
\end{dfn}

There is the canonical projection $\calL\calG\to X$ (resp. $\calL^+\calG\to X$), and the construction is functorial in $(\calG, X)$ in the obvious sense.

\begin{rmk}\label{connectrmk}
The connection to the loop groups from \S 1.1 is as follows. Let $x\in X(K)$ for any field extension $K$ of $k$, and choose a local parameter $t_x$ in $x$. We let $F_x=K\rpot{t_x}$, with ring of integers $\calO_{F_x}=K\pot{t_x}$. Then as functors on the category of $K$-algebras
\[(\calL\calG)_x\;=\;L\calG_{F_x}\;\;\;\;\text{and}\;\;\;\;(\calL^+\calG)_x\;=\;L^+\calG_{\calO_x}.\]
\end{rmk}

Let $\calG$ be an affine group scheme over $X$. For $i\geq 0$ let $\calG_i$ be the functor on the category of $k$-algebras
\[R\;\longmapsto\;\{(x,g)\;|\;x\in X(R),\; g\in\calG(\Ga_{x,i})\},\]
where $\Ga_{x,i}$ denotes the $i$-th infinitesimal neighbourhood of $\Ga_x\hookto X_R$. If $x\in X(R)$, then $\calG_i\otimes_X R$ is the Weil restriction of $\calG_R\times_{X_R}\Ga_{x,i}$ along the finite flat morphism $\Ga_{x,i}\to \Spec(R)$. Hence, $\calG_i\to X$ is representable by an affine group scheme by \cite[\S 7.6 Proof of Theorem 4]{BLR}. As $i$ varies, the $\calG_i$ form an inverse system with affine transition maps. In particular, $\varprojlim_i\calG_i$ is representable by an affine group scheme over $X$, and the canonical morphism of group functors
\begin{equation}\label{invlim}
\calL^+\calG\;\overset{\simeq}{\longto}\; \varprojlim_i\calG_i
\end{equation}
is an isomorphism. 

\begin{lem}\label{groupscheme} Let $\calG$ be a flat affine group scheme of finite type over $X$ .\smallskip\\
i) The functor $\calL\calG$ is representable by an ind-affine ind-group scheme over $X$.\smallskip\\
ii) The functor $\calL^+\calG$ is representable by a closed affine subgroup scheme of $\calL\calG$. If $\calG$ is smooth (resp. is smooth and has geometrically connected fibers) over $X$, then $\calL^+\calG$ is reduced and flat (resp. reduced, flat and has geometrically connected fibers) over $X$.
\end{lem}
\begin{proof}
The representability assertions are due to Heinloth \cite[Proposition 2]{He}. First assume that $\calG$ is smooth. Then, for all $i\geq 0$, the group schemes $\calG_i\to X$ are smooth by \cite[\S 7.6 Proposition 5]{BLR}. This implies that $\calL^+\calG$ is reduced and flat over $X$ by \eqref{invlim}. If in addition $\calG$ has geometrically connected fibers, then it follows easily by induction on $i$ that $\calG_i\to X$ has geometrically connected fibers. Again by \eqref{invlim}, this implies that $\calL^+\calG\to X$ has geometrically connected fibers.  
\end{proof}

Let $\calG$ be a smooth affine group scheme over $X$. Let $\calL\calG/\calL^+\calG$ be the fpqc-quotient. We construct a morphism of fpqc-sheaves 
\[\ev\colon \Gr(\calG,X)\;\longto\;\calL\calG/\calL^+\calG\]
as follows. Let $(x,\calF,\be)\in\Gr(\calG,X)(T)$ with $T$ affine. First assume that $\calF|_{\hGa_x}$ is trivial, and choose some trivialization $\be\colon \calF|_{\hGa_x}\simeq\calF^0|_{\hGa_x}$. Then the class of $(\al\circ\be^{-1})(1)$ in $\calG(\hGa_x^o)/\calG(\hGa_x)$ is independent of $\be$, and defines $\ev$ in this case. To construct $\ev$ in general, it is enough to observe that if $\calF$ is torsor on the Laurent series ring $R\pot{t}$ under a smooth group scheme, then $\calF\otimes_{R\pot{t}}R'\pot{t}$ is trivial for some $R\to R'$ fpqc. Indeed, $\calF|_{t=0}$ is trivial over a faithfully flat extension $R\to R'$ giving rise to a section $s\in\calF(R')$. By the smoothness, $s$ extends to a formal section over $\Spf(R'\pot{t})$ which is algebraic by Grothendieck's algebraization theorem. Hence, $\calF|_{R'\pot{t}}$ is trivial. 

\begin{lem}\label{evlem} 
Let $\calG$ be a smooth affine group scheme over $X$. The morphism $\ev\colon  \Gr(\calG,X)\to\calL\calG/\calL^+\calG$ is an isomorphism of fpqc-sheaves which is functorial in $\calG$ and $X$.
\end{lem}
\begin{proof}
Because $\Gr(\calG,X)$ is of ind-finite type, we may test on noetherian $k$-algebras that $\ev$ is an isomorphism. This follows from Proposition 4 in \cite{He}, cf. Remark \ref{Heinrmk}. The functorialities are easy to check.
\end{proof}

By Lemma \ref{evlem}, there is an action of the loop group 
\begin{equation}\label{Globact}
\calL\calG\times_X\Gr(\calG,X)\longto \Gr(\calG,X).
\end{equation}
If we restrict the action to $\calL^+\calG$, then it factors on each orbit through a smooth affine quotient of $\calL^+\calG$ in the following sense.

\begin{lem} \label{actlem} Let $\calG$ be a smooth affine group scheme over $X$. Let $\calL^+\calG\simeq \varprojlim_i\calG_i$ be as in \eqref{invlim}. Let $T$ be a quasi-compact $X$-scheme, and let $\mu\colon  T\to \Gr(\calG,X)$ be a morphism over $X$. Then the $T$-morphism $\calL^+\calG_T\to \Gr(\calG,X)_T$, $g\mapsto g\mu$ factors through some $\calG_{i,T}$. 
\end{lem}
\begin{proof}
Since $T$ is quasi-compact, the morphism $\mu$ factors through some closed subscheme of $\Gr(\calG,X)$. Now the lemma follows by reduction to the case of $\GL_n$ as in \cite[Cor. 2.7]{Ri2}.
\end{proof}

\begin{cor}\label{group1} Let $\calG$ be a smooth affine group scheme over $X$.\smallskip\\
i) If $x\in X(K)$ for any extension $K$ of $k$, let $F_x=K\rpot{t_x}$ and $\calO_{F_x}=K\pot{t_x}$ where $t_x$ is a local parameter at $x$. The fiber of \eqref{Globact} over $x$ is isomorphic to the action morphism
\[L\calG_{F_x}\times_{K}L\calG_{F_x}/L^+\calG_{\calO_x}\;\longto\; L\calG_{F_x}/L^+\calG_{\calO_x},\]
functorial in $\calG$.\smallskip\\
ii) If $f\colon Y\to X$ is \'etale, then $f^*\colon \Gr(\calG,X)\times_XY\to\Gr(\calG_Y,Y)$ is an isomorphism.
\end{cor}
\begin{proof}
Part i) is a consequence of Lemma \ref{evlem}. Part ii) is proven in \cite[Lemma 3.2]{Z}.
\end{proof}

\begin{rmk} If $\calG$ is \emph{parahoric} (cf. Definition \ref{parahoricdfn} below), then the fibers of $\Gr(\calG,X)\to X$ are affine flag varieties, cf. \S 1.1 and Remark \ref{connectrmk}.
\end{rmk}

\begin{lem}\label{reductivelem} Let $\calG$ be a smooth affine group scheme over $X$.\smallskip\\
i) If $\calG=G\times_kX$ is constant, then $\Gr(G,X)\to X$ is Zariski locally on $X$ constant.\smallskip\\
ii) If $\calG$ is connected reductive, then $\Gr(\calG,X)\to X$ is ind-proper.
\end{lem}
\begin{proof}
The Grassmannian $\Gr(G,X)\to X$ is constant for $X=\bbA^1_k$, cf. \cite[Remark 2.19]{Ri2}. But Zariski locally on $X$, there exists a finite \'etale morphism $X\to \bbA^1_k$ which implies i) by Corollary \ref{group1} ii). For part ii), if $\calG$ is split reductive, then the existence of Chevalley models shows that there is a faithful representation $\calG\hookto \GL_{n,X}$ such that the fppf-quotient $\GL_{n,X}/\calG$ is affine. Use Lemma \ref{affgrass1} i), and the ind-properness of $\Gr(\GL_n,X)\to X$ to conclude that $\Gr(\calG,X)\to X$ is ind-proper in this case. The lemma follows from the fact that every connected reductive group is split \'etale locally, and from Corollary \ref{group1} ii).
\end{proof}

The following general lemma is needed in \S 1.4.

\begin{lem}\label{generalnonsense} Let $\iota\colon  \tilde{X}\to X$ be a finite flat surjection of smooth connected curves over $k$, and let $\tilde{G}$ be an affine group scheme of finite type over $\tilde{X}$. Then $\calG=\Res_{\tilde{X}/X}(\tilde{G})$ is an affine group scheme of finite type, and the canonical morphism of functors over $X$
\[\Res_{\tilde{X}/X}(\Gr(\tilde{\calG},\tilde{X})) \;\longto\; \Gr(\calG,X)\]
is an isomorphism.
\end{lem}
\begin{proof}
By \cite[\S 7.6 Proof of Theorem 4]{BLR}, the sheaf of groups $\calG$ is affine and of finite type. The lemma follows from fpqc-descent for affine schemes, cf. \cite[\S 2.6]{Levin}. \end{proof}

\subsection{Geometric characterization of parahoric groups} Let $X$ be a smooth connected curve over a field $k$ which is either finite or algebraically closed. Let $|X|$ be the set of closed points in $X$. For any $x\in |X|$, we denote by $\calO_x$ the completed local ring of $X$ in $x$, and by $F_x$ its field of fractions.

\begin{dfn}\label{parahoricdfn}
A smooth affine group scheme $\calG$ over $X$ with geometrically connected fibers is called \emph{parahoric} if its generic fiber is reductive and for all $x\in |X|$, the group $\calG_{\calO_x}$ is a parahoric group scheme in the sense of \cite{BT2}.
\end{dfn}

\begin{thm}\label{indproj}
Let $\calG$ be a smooth affine group scheme over $X$ with geometrically connected fibers. \smallskip\\
i) Then $\calG$ is parahoric if and only if the fibers of $\Gr(\calG,X)\to X$ are ind-proper. \smallskip\\ 
ii) In case i), $\Gr(\calG,X)\to X$ is ind-proper, and Zariski locally on $X$ ind-projective.
\end{thm}

The proof of ii) is explained in \S 1.5 below, based on \S 1.4, in which we study the case of tori.

\begin{proof}[Proof of Theorem \ref{indproj} i).] Let $\calG$ be parahoric group scheme over the curve $X$. By Corollary \ref{group1}, the generic fiber $\Gr(\calG,X)_\eta$ is the affine Grassmannian associated with the reductive group $\calG_\eta$, and hence ind-proper. Let $x\in |X|$. Then $\Fl=\Gr(\calG,X)_x$ is the affine flag variety associated with the parahoric group scheme $\calG_{\calO_x}$ which is ind-proper, cf. Corollary \ref{flagindproper} i).
\vskip 0.1 cm

Conversely assume that the fibers of $\Gr(\calG,X)\to X$ are ind-proper. In particular, $\Gr(\calG,X)_{\oeta}$ is ind-proper where $\oeta$ denotes a geometric generic point. But $\Gr(\calG,X)_{\oeta}$ is the affine Grassmannian associated with the linear algebraic group $\calG_{\oeta}$, and it follows from the argument in \cite[Appendix]{Ga} that $\Gr(\calG,X)_{\oeta}$ is ind-proper if and only if $\calG_{\oeta}$ is reductive. Fix $x\in |X|$. We need to show that $\calG_{O_x}$ is a parahoric group scheme. We may assume that $k$ is algebraically closed. Let $\calO=\calO_{x}$, $F=F_x$, and choose a uniformizer $t=t_x$ in $\calO$. Let $G=\calG_{F}$ and $\calG=\calG_{\calO}$. The subgroup $\calG(\calO)\subset G(F)$ is bounded for the $t$-adic topology, and hence is contained in some maximal bounded subgroup $\calG(\calO)_{\max}$. By \cite[\S 3.2]{Ti}, the group $\calG(\calO)_{\max}=\calG_{\al}(\calO)$ is the stabilizer of some point $\al$ of the Bruhat-Tits building. By \cite[Proposition 1.7.6]{BT2}, the inclusion $\calG(\calO)\subset \calG_{\al}(\calO)$ prolongs into a morphism $\tau\colon  \calG\to \calG_{\al}$ of group schemes over $\calO$. Since the special fiber of $\calG$ is connected, we may replace $\calG_{\al}$ by its neutral component, and hence assume that $\calG_{\al}$ is the parahoric group scheme associated with the facet containing ${\al}$. Consider the induced morphism 
\[\tau_*\colon  LG/L^+\calG\;\longto\; LG/L^+\calG_{\al}.\]
The fiber over the base point $e_{\calG_{\al}}$ in $LG/L^+\calG_{\al}$ is the fpqc-quotient $L^+\calG_{\al}/L^+\calG$, and this is representable by an ind-proper ind-scheme since $LG/L^+\calG$ is ind-proper by assumption. Let $\bar{\calG}_{\al}$ be the maximal reductive quotient of the special fiber $\calG_{\al}\otimes k$, and let $\bar{\calG}'$ be the scheme theoretic image of $L^+\calG$ under the morphism $L^+\calG_{\al}\to \calG_{\al}\otimes k\to \bar{\calG}_{\al}$, $t\mapsto 0$. The quotient $\bar{\calG}_{\al}/\bar{\calG}'$ is representable by a separated scheme of finite type over $k$, and the morphism
\[L^+\calG_{\al}/L^+\calG\;\longto\; \bar{\calG}_{\al}/\bar{\calG}'\]
is surjective. Hence $\bar{\calG}_{\al}/\bar{\calG}'$ is proper, i.e. $\bar{\calG}'$ is a parabolic subgroup of $\bar{\calG}_{\al}$. The preimage of $\bar{\calG}'(k)$ under the reduction $\calG_{\al}(\calO)\to\bar{\calG}_{\al}(k)$ is by \cite[3.5.1]{Ti} a parahoric subgroup of $G(F)$ associated with some facet $\fraka$. Let $\calG_\fraka$ be the corresponding parahoric group scheme. The morphism $\calG\to \calG_{\al}$ factorizes as $\calG\to \calG_\fraka\to \calG_{\al}$. For $i\geq 0$, let 
\[\calG_i\;=\;\Res_{(\calO/t^{i+1})/k}(\calG\otimes \calO/t^{i+1})\]
be the Weil restriction. Then $\calG_i$ is a connected smooth affine group scheme over $k$, and $L^+\calG\simeq \varprojlim_i\calG_i$. Let $\calG_{\fraka,i}$ be defined analogously. One shows as above that $\calG_{\fraka,i}/\calG'_i$ is representable by a connected proper scheme. Because the kernel of $L^+\calG_\fraka\to \bar{\calG}_\fraka$ is pro-unipotent, it follows that $\calG_{\fraka,i}=\calG'_i$ for all $i\geq 0$. Hence, the morphism $\calG(\calO/t^{i+1})\to \calG_\fraka(\calO/t^{i+1})$ is surjective, and the kernel is finite by dimension reasons. Using Mittag-Leffler we see that $\calG(\calO)=\calG_\fraka(\calO)$ is parahoric.
\end{proof}

\subsection{The case of tori} In this subsection, we assume that the field $k$ is algebraically closed. Let $\calT$ be a smooth affine group scheme with geometrically connected fibers over $X$. If the generic fiber $\calT_\eta$ is a torus, then $\calT$ is a parahoric group scheme if and only if $\calT$ is the connected lft-Ner\'on model of $\calT_\eta$, cf. \cite{PR}. See \cite[\S 10.1-10.3]{BLR} for the definition of the lft-Ner\'on model.

\begin{lem}\label{torlem} 
Let $\calT$ be a parahoric group scheme over $X$. Then the generic fiber $\calT_\eta$ is a torus if and only if $\Gr(\calT,X)\to X$ is ind-finite.
\end{lem}
\begin{proof} If $\calT_\eta$ is not a torus, then $\Gr(\calT,X)_\oeta$ is not ind-finite. Conversely assume that $\calT_\eta$ is a torus. The Kottwitz morphism implies that $\Gr(\calT,X)\to X$ is ind-quasi-finite, and it is enough to show the ind-properness.

There is a non-empty open subset $U\subset X$ such that $\calT\times_XU$ is reductive which implies the ind-properness of $\Gr(\calT,X)|_U\to U$. For $x\in |X|$, let $\calO_x$ be the complete local ring at $x$ with fraction field $F_x$. By fpqc-descent we are reduced to showing that $\Gr(\calT,X)\otimes\calO_x$ is ind-finite for every $x\in |X|$. Fix $x\in |X|$. \smallskip\\
\emph{The induced case.} Assume first that $\calT\simeq\prod_{i=1}^k\Res_{X_i/X}(\bbG_m)^{n_i}$ where $X_i\to X$ are finite flat generically \'etale\footnote{Otherwise, the generic fiber of $\calT$ is not reductive.} surjections of smooth connected curves. By Lemma \ref{generalnonsense}, we are reduced to considering the case of 
\[\Res_{\tilde{X}/X}(\Gr(\bbG_m,\tilde{X})).\]
Then on complete local rings 
\[\calO_{\tilde{X}}\otimes \calO_x\;\simeq\;\prod_{\tilde{x}\mapsto x} \calO_{\tilde{x}}\]
because the field extension on the generic points is separable. Since Weil restriction is compatible with base change, we are reduced to proving that $\Res_{\calO_{\tilde{x}}/\calO_x}(\Gr(\bbG_m,\tilde{X})\otimes\calO_{\tilde{x}})$ is ind-finite. This follows from Lemma \ref{reductivelem} i) using that the reduced locus of $\Gr(\bbG_m,\tilde{X})\otimes\calO_{\tilde{x}}$ is the constant scheme $\underline{\bbZ}$, and the tautological equality $\Res_{\calO_{\tilde{x}}/\calO_x}(\Spec(\calO_{\tilde{x}}))=\Spec(\calO_x)$.

In the general case, after replacing $X$ by a Zariski neighbourhood of $x$ there exists a morphism of parahoric group schemes 
\begin{equation}\label{inducedresol}
\pi\colon  \calT'\;\longto\;\calT,
\end{equation}
such that $\calT'\simeq\prod_{i=1}^k\Res_{X_i/X}(\bbG_m)^{n_i}$ is induced as above, and such that $\pi_*\colon \calY'\to\calY$ is surjective where $\calY=\Gr(\calT,X)\otimes\calO_x$ (resp. $\calY'=\Gr(\calT',X)\otimes\calO_x$). The morphism \eqref{inducedresol} is constructed as follows. By \cite{CT}, there exists an exact sequence of tori $1\to T''\to T'\to\calT_\eta\to 1$ over $\eta$ where $T'$ is induced and $T''$ is flasque. As $X_i$ we take the normalization of $X$ in the field extensions $\eta_i/\eta$ defining $T'$. This allows to define $\calT'$ in a Zariski neighbourhood of $x$ such that $\calT'_\eta=T'$. Then $\calT'$ is the connected component of the lft-Ner\'on model of $\calT'_\eta$, i.e. $\calT'$ is parahoric. Using the Ner\'on mapping property for $\calT$, the group morphism on generic fibers $\calT'_\eta\to\calT_\eta$ extends to a group morphism $\pi\colon \calT'\to\calT$. Then $\pi_*\colon \calY'\to\calY$ is surjective over $\sF_x$ because $\calT'_{\sF_x}\to\calT_{\sF_x}$ is a surjection of split tori, and $\pi_*$ is surjective on $k$-points by \cite[\S 7]{Kott}: \\ Indeed $\calY(k)=\calT(F_x)/\calT(\calO_x)$ (resp. $\calY'(k)=\calT'(F_x)/\calT'(\calO_x)$) and the Kottwitz morphism gives a commutative diagram of abelian groups
\[
\begin{tikzpicture}[baseline=(current  bounding  box.center)]
\matrix(a)[matrix of math nodes, 
row sep=1.5em, column sep=2em, 
text height=1.5ex, text depth=0.45ex] 
{\calY'(k)& \calY(k) \\ 
X_*(\calT'_{F_x})_{I_x} &X_*(\calT_{F_x})_{I_x}& 0, \\}; 
\path[->](a-1-1) edge node[above] {$\pi_*(k)$} (a-1-2);
\path[->](a-2-1) edge (a-2-2);
\path[->](a-1-1) edge node[right] {$\simeq$} (a-2-1);
\path[->](a-1-2) edge node[right] {$\simeq$} (a-2-2);
\path[->](a-2-2) edge (a-2-3);
\end{tikzpicture}
\]
where $I_x$ denotes the inertia group of $F_x$. The lower row is surjective because $T''$ is flasque, cf. \cite[\S 7 (7.2.5)]{Kott}. This shows the surjectivity of $\pi_*\colon \calY'\to\calY$. Since $\calY'$ is ind-proper and $\calY$ is separated, $\calY$ is also ind-proper. The lemma follows. 
\end{proof}

Let $\calT$ be a parahoric group scheme whose generic fiber is a torus. Fix $x\in |X|$, and let $\Gr_\calT=\Gr(\calT,X)\otimes\calO_x$. Since $\Gr_\calT$ is ind-proper, there is a specialization map
\begin{equation}\label{specimap}
\spe\colon  \Gr_\calT(\sF) \longto \Gr_\calT(k).
\end{equation}
Note that $\Gr_\calT$ is a sheaf of groups because $\calT$ is commutative, and that $\spe$ is a group morphism. The generic fiber $\Gr_{\calT,\sF}$ is by Corollary \ref{group1} equal to the affine Grassmannian associated with the split torus $\calT_\sF$, and hence $\Gr_\calT(\sF)\simeq X_*(T)$ as groups. The Kottwitz morphism $\kappa_T\colon \Gr_\calT(k)\to X_*(T)_I$ is an isomorphism, cf. \cite[\S 7]{Kott}. The following lemma is proven in \cite[Proposition 3.4]{Z} in the tamely ramified case.

\begin{lem}\label{specitorus}
There is a commutative diagram of abelian groups
\[
\begin{tikzpicture}[baseline=(current  bounding  box.center)]
\matrix(a)[matrix of math nodes, 
row sep=1.5em, column sep=2em, 
text height=1.5ex, text depth=0.45ex] 
{\Gr_\calT(\sF)& \Gr_\calT(k)\\ 
X_*(T) &X_*(T)_I \\}; 
\path[->](a-1-1) edge node[above] {$\spe$} (a-1-2);
\path[->](a-2-1) edge (a-2-2);
\path[->](a-1-1) edge node[right] {$\simeq$} (a-2-1);
\path[->](a-1-2) edge node[right] {$\simeq$} node[left] {$\kappa_T$} (a-2-2);
\end{tikzpicture}
\]
where $X_*(T)\to X_*(T)_I$ is the canonical projection. 
\end{lem}
\begin{proof}
We show that $X_*(T)\simeq \Gr_\calT(\sF)\to \Gr_\calT(k)\simeq X_*(T)_I$ is the canonical projection. Let $T''\to T'\to T\to 1$ be a resolution of $T$ by induced tori $T', T''$ as in \cite[\S 7 (7.2.5)]{Kott}. Using the argument in the proof of Lemma \ref{torlem}, we obtain an exact sequence $\Gr_{\calT''}\to\Gr_{\calT'}\to\Gr_\calT\to 1$ of sheaf of groups over $\calO_x$. This gives a commutative diagram of abelian groups
\[
\begin{tikzpicture}[baseline=(current  bounding  box.center)]
\matrix(a)[matrix of math nodes, 
row sep=1.5em, column sep=2em, 
text height=1.5ex, text depth=0.45ex] 
{\Gr_{\calT''}(\sF)&\Gr_{\calT'}(\sF) & \Gr_{\calT}(\sF)  & 0\\ 
\Gr_{\calT''}(k)&\Gr_{\calT'}(k)& \Gr_{\calT}(k)  & 0\\}; 
\path[->](a-1-1) edge node[right] {$\spe$} (a-2-1); 
\path[->](a-1-2) edge (a-1-3); 
\path[->](a-1-3) edge node[right] {$\spe$}  (a-2-3); 
\path[->](a-1-2) edge node[right] {$\spe$} (a-2-2);
\path[->](a-2-2) edge  (a-2-3); 
\path[->](a-1-3) edge  (a-1-4); 
\path[->](a-2-3) edge (a-2-4); 
\path[->](a-1-1) edge (a-1-2);
\path[->](a-2-1) edge (a-2-2);
\end{tikzpicture}
\]
with exact rows, and we may reduce to the case that $T$ is induced. Let $A\subset T$ be the maximal split subtorus. Since $A$ is defined over $k$, we see that $\Gr_{A,\red}\simeq \underline{X_*(A)}$ is the constant group scheme over $\calO_x$ associated with $X_*(A)$, and hence $\spe$ is just the identity. Consider the commutative diagram of abelian groups 
\[
\begin{tikzpicture}[baseline=(current  bounding  box.center)]
\matrix(a)[matrix of math nodes, 
row sep=1.5em, column sep=2em, 
text height=1.5ex, text depth=0.45ex] 
{0&\Gr_A(\sF)& \Gr_T(\sF)\\ 
 0&\Gr_A(k) & \Gr_\calT(k). \\}; 
\path[->](a-1-1) edge (a-1-2); 
\path[->](a-2-1) edge (a-2-2);
\path[->](a-1-2) edge (a-1-3);
\path[->](a-2-2) edge (a-2-3);
\path[->](a-1-2) edge node[right] {$\spe$} (a-2-2);
\path[->](a-1-3) edge node[right] {$\spe$} (a-2-3);
\end{tikzpicture}
\]
with exact rows. Note that $\Gr_A(k)$, $\Gr_\calT(k)$ have the same rank. Hence, the composition $X_*(T)_\bbQ\simeq \Gr_\calT(\sF)_\bbQ\to \Gr_\calT(k)_\bbQ\simeq X_*(T)_{\bbQ,I}$ rationally is the canonical projection. The lemma follows from the fact that $\Gr_\calT(k)$ is torsionfree because $T$ is induced.
\end{proof}

\subsection{Proof of the ind-projectivity} Let us prove Theorem \ref{indproj} ii). Let $\calG$ be a parahoric group scheme over $X$. Then $\Gr(\calG,X)\to X$ is fiberwise ind-proper, and we need to show that it is ind-proper and, Zariski locally on $X$, ind-projective. Let $\iota\colon \calG\to \GL(\calE)$ be a faithful representation such that $\GL(\calE)/\calG$ is quasi-affine (cf. \cite[\S 2 Example (1)]{He}), where $\calE$ is some vector bundle on $X$. Then $\iota_*\colon  \Gr(\calG,X)\to \Gr(\GL(\calE),X)$ is representable by an immersion, cf. Lemma \ref{affgrass1} ii). It is enough to prove that $\Gr(\calG,X)\to X$ is ind-proper. Since $\calG_\eta$ is reductive, there is a non-empty open subset $U\subset X$ such that $\calG_U$ is reductive. Lemma \ref{reductivelem} ii) shows that $\Gr(\calG,X)|_U\to U$ is ind-proper. By fpqc-descent, we are reduced to proving that $\Gr(\calG,X)\otimes \calO_{x}$ is ind-proper for every $x\in |X|$. We may assume that $k$ is algebraically closed. Fix $x\in |X|$, and let $\calO=\calO_x$ and $F=F_x$. We claim that the reduced locus of $\Gr(\calG,X)\otimes \calO$ can be written as
\begin{equation}\label{globindrep}
(\Gr(\calG,X)\otimes \calO)_{\text{red}}\;=\;\varinjlim_{\ga\in J}M_\ga,
\end{equation}
where $M_\ga$ are closed subschemes with the following properties: The $M_\ga$ are separated schemes of finite type over $\calO$ such that $M_\ga\to \Spec(\calO)$ is surjective flat, and the generic fiber $M_\ga\otimes F$ is connected. Since $\Gr(\calG,X)$ is fiberwise ind-proper, Lemma \ref{funnylem} below reduces us to constructing the ind-presentation \eqref{globindrep}. Let $J$ be the set of Galois orbits of $(\calL^+\calG)_{\sF}$-orbits in $\Gr(\calG,X)\otimes \sF$. Then every $\ga\in J$ defines a connected closed reduced subscheme $M_{\ga,F}$ of $\Gr(\calG,X)\otimes F$. Let $M_\ga$ be the scheme theoretic closure of $M_{\ga,F}$ in $\Gr(\calG,X)\otimes\calO$. Then $M_\ga$ is a flat reduced $(\calL^+\calG)_\calO$-equivariant closed subscheme of $\Gr(\calG,X)\otimes \calO$ (use Lemma \ref{actlem}). Let $\calT_\eta\subset \calG_\eta$ be a maximal torus, and denote by $\calT$ the scheme theoretic closure in $\calG$. Then $\calT$ is a parahoric group scheme over $X$. By Lemma \ref{torlem}, $\Gr(\calT,X)$ is ind-proper, and hence $\Gr(\calT,X)\to \Gr(\calG,X)$ is a closed immersion. The ind-presentation \eqref{globindrep} follows from Lemma \ref{specitorus} noting that the affine flag variety in the special fiber of $\Gr(\calG,X)\otimes\calO$ is covered by the orbit closures of the translation elements. This proves Theorem \ref{indproj} ii).
\hfill\ensuremath{\Box} 

\vspace{0.5cm}

In the proof above we used the following lemma, a special case of [EGA IV, 15.7.10] under the hypothesis that the geometric fibers are connected. If the base is a complete discrete valuation ring, the hypothesis on the fibers can be weakened. 

\begin{lem}\label{funnylem}
Let $Y$ be a separated scheme of finite type over a complete discrete valuation ring $\calO$. Assume that $Y\to \Spec(\calO)$ is surjective flat, and that the generic fiber $Y_\eta$ is connected. Then $Y$ is proper if and only if the fibers $Y_\eta$ and $Y_s$ are proper.
\end{lem}
\begin{proof}
Let $Y_\eta$ and $Y_s$ be proper. Note that by assumption both are non-empty. To prove properness we may assume that $Y$ is reduced. Let $\iota\colon  Y\to \bar{Y}$ be the Nagata compactification over $\calO$, i.e. $\bar{Y}$ is proper over $\calO$, and $\iota$ is an open immersion. Now replace $\bar{Y}$ by the scheme theoretic closure of $\iota(Y)$ in $\bar{Y}$. We claim that the open immersion $\iota\colon Y\to \bar{Y}$ is an isomorphism. Because $Y_\eta$ is proper, $\iota_\eta$ is an isomorphism onto a connected component of $\bar{Y}_\eta$. But since $Y$ is flat, the generic fiber $Y_\eta$ is open dense in $Y$ and hence, $\iota(Y_\eta)$ is open dense in $\bar{Y}$. It follows that $\iota_\eta$ is an isomorphism. Now since $\bar{Y}_\eta$ is open dense in $\bar{Y}$, and $\bar{Y}$ is reduced, it follows that $\pi\colon \bar{Y}\to \Spec(\calO)$ is flat. Hence, $\pi_*\calO_{\bar{Y}}$ is a finite free $\calO$-module of rank $1$ because $\bar{Y}_\eta$ is connected. By proper base change we have $\dim(H^0(\bar{Y}_s,\calO_{\bar{Y}_s}))=1$, and hence $\bar{Y}_s$ is connected. It follows that $\iota_s\colon Y_s\to\bar{Y}_s$ is an isomorphism because it is open and proper, $Y_s\not =\varnothing$ and $\bar{Y}_s$ is connected. All in all, $\iota$ is a fiberwise isomorphism between flat schemes. The lemma follows.   
\end{proof}

\section{The BD-Grassmannian associated with a facet}\label{globSchub} 
We define the global Schubert varieties which may be seen as analogues of local models in equal characteristic, cf. \cite{PRS}. These are introduced by Zhu \cite{Z} in the tamely ramified case. The results of this paragraph are used in the proof of Theorem B of the introduction in the next section.

Let $k$ be either a finite or an algebraically closed field, and let $G$ be a connected reductive group over the Laurent power series field $F=k\rpot{t}$. Let $\fraka$ be a facet in the extended Bruhat-Tits building $\scrB(G,F)$, and denote by $\calG=\calG_\fraka$ the associated parahoric group scheme over $\calO_F=k\pot{t}$. Then $\calG$ is a smooth affine group scheme with geometrically connected fibers. In a first step, we aim to descend $\calG$ to some smooth connected pointed curve $(X,x)$ with $\calO_x=\calO_F$. Even better:

\begin{lem} Let $\calH$ be a flat affine $\calO_F$-group scheme of finite type with $\calH_F$ connected reductive. Then $\calH$ descends to some open curve, i.e. there exists a smooth connected pointed curve $(X,x)$ with $\calO_x=\calO_F$ and a flat affine $X$-group scheme $\calH_X$ such that\smallskip\\
{\rm i)} the restriction $\calH_X|_{(X\bslash x)}$ is connected reductive; \smallskip \\
{\rm ii)} $\calH_X\otimes \calO_x \simeq \calH$ as $\calO_F$-groups.\smallskip\\
Moreover, if $\calH$ is smooth (resp. has geometrically connected fibers), then $\calH_X$ is smooth (resp. has geometrically connected fibers). 
\end{lem}
\begin{proof} 
Pick any global field $E$ and a place $v$ with $E_v=F$, i.e. $\calO_{E_v}=\calO_F$ on valuation rings. Denote by $\calO_{E,v}$ the algebraic local ring of $E$ at $v$, and let $\calO_{E,v}^h$ be its henselization, a henselian discrete valuation ring. Then on completions 
\[\hat{\calO}^h_{E,v}=\hat{\calO}_{E,v}=\calO_{E_v}= \calO_{F},\] 
and by Corollary \ref{reductivedescent2} the $\calO_F$-group $\calH$ descends to a flat affine $\calO_{E,v}^h$-group $\calH'$ of finite type. The ring $\calO_{E,v}^h$ is the direct limit over all finite \'etale $\calO_{E,v}$-algebras $B$ with $B/\frakm_vB=\calO_{E,v}/\frakm_v$ where $\frakm_v$ is the maximal ideal of $\calO_{E,v}$. Since $\calH'$ is of finite type, it is defined over a finite \'etale local $\calO_{E,v}$-algebra, i.e. over the algebraic local ring $\calO_{E',v'}$ of a finite separable extension $E'/E$ with place $v'$ such that $\hat{\calO}_{E',v'}=\hat{\calO}_{E,v}$. Hence, $\calH'$ descends to $\calH_X$ in a small neighborhood $X$ of $x:= v'\in \Spec(\calO_{E'})$. This proves the existence of $\calH_X$ with properties i) and ii) above. If $\calH$ is smooth, then $\calH_X$ is smooth by fpqc-descent. Geometrically connectedness of the fibers over some point $x'\in X$ may be tested on any algebraically closed field lying over $x'$. The lemma follows. 
\end{proof}

\begin{rmk}
If $G$ is tamely ramified over $F$, Zhu \cite[\S 2]{RZ} extends the parahoric group scheme over a pointed $\bbA^1_k$. In general, the curve $X$  seems to depend on the reductive group $G$. As B. Conrad pointed out to me, one could also ask the finer question whether $G$ descents to a specific global model $(E,v)$ of  the local field $F$ (instead of passing to some unknown finite separable $(E',v')$). Already for tori this does not seem to be obvious, if true at all. 
\end{rmk}

Denote by $\calG_X$ the extension of $\calG$ over some $(X,x)$ with $\calO_x=\calO_F$. Then $\calG_X$ is a smooth affine $X$-group scheme with geometrically connected fibers and is parahoric in the sense of Definition \ref{parahoricdfn}. 

Let $\sF$ be the completion of a separable closure of $F$ with valuation subring $\calO_\sF$. Let $S=\Spec(\calO_F)$, $\bar{S}=\Spec(\calO_\sF)$ with generic points $\eta,\oeta$ and special points $s,\os$. This leads to the $6$-tuple $(S,\bar{S},\eta,\oeta,s,\os)$.

\begin{dfn}\label{BDGrass}
i) The \emph{global (resp. global positive) loop group $\calL\calG$ (resp. $\calL^+\calG$) associated with $\fraka$} is the ind-group scheme (resp. group scheme) over $S$
\[\calL\calG\defined \calL\calG_X\times_X S \hspace{0.5cm}\text{(resp.}\;\; \calL^+\calG\defined \calL^+\calG_X\times_X S).\]
ii) The \emph{BD-Grassmannian $\Gr_\fraka$ associated with $\fraka$} is the ind-scheme over $S$
\[\Gr_\fraka\defined \Gr(\calG_X,X)\times_X S.\]
\end{dfn}

Note that the definitions do not depend on the choice of $X$. There is a left action
\begin{equation}\label{actBD}
\calL\calG\times_S\Gr_\fraka\;\longto\; \Gr_\fraka.
\end{equation}
Let us discuss the generic and the special fiber of \eqref{actBD}. Let $L_zG$ (resp. $L_z^+G$) be the functor on the category of $F$-algebras $L_zG\colon R\mapsto G(R\rpot{z})$ (resp. $L_z^+G\colon R\mapsto G(\pot{z})$) where $z$ is an additional variable. The \emph{affine Grassmannian $\Gr_G$} is the fpqc-quotient $\Gr_G=L_zG/L_z^+G$.  There is a left action
\[L_zG\times_F \Gr_G\;\longto\;\Gr_G.\]
Recall from \S \ref{affflagpara} the following objects: Let $LG$ (resp. $L^+\calG$) be the functor on the category of $k$-algebras $LG\colon R\mapsto G(R\rpot{t})$ (resp. $L^+\calG\colon R\mapsto \calG(R\pot{t})$). The affine flag variety $\Fl_\fraka$ is the fpqc-quotient $\Fl_\fraka=LG/L^+\calG$. There is a left action
\[LG\times_k \Fl_\fraka\;\longto\;\Fl_\fraka.\]
\begin{lem}\label{satcat1} The ind-scheme $\Gr_\fraka\to S$ is ind-projective, and \smallskip\\
i) the generic fiber $\Gr_{\fraka,\eta}$ is equivariantly isomorphic to $\Gr_G$.\smallskip\\
ii) the special fiber $\Gr_{\fraka,s}$ is equivariantly isomorphic to $\Fl_\fraka$.
\end{lem}
\begin{proof}
This follows from Theorem \ref{indproj} and Corollary \ref{group1}.
\end{proof}

Next we introduce the global Schubert varieties which are reduced $\calL^+\calG$-orbit closures in $\Gr_\fraka$, cf. \cite{Z}. Let $A\subset G$ be a maximal $F$-split torus such that $\fraka\subset \scrA(G,A,F)$. Let $\breve{A}$ be a maximal $\nF$-split torus defined over $F$ with $A\subset \breve{A}$, cf. \cite{BT2}. Let $T=Z_G(\breve{A})$ be the centralizer which is a maximal torus since $G_\nF$ is quasi-split by Steinberg's Theorem. Let $\calT$ be the scheme theoretic closure of $T$ in $\calG$, which is a parahoric group scheme by Appendix \ref{levilemapp}, Remark \ref{levirem}. Let $\Gr_\calT$ be the BD-Grassmannian over $\calO_F$. Every $\mu\in X_*(T)$ determines a unique point
\[\tilde{\mu}\colon  \bar{S}\;\longto\; \Gr_\fraka.\]
Indeed, we have $X_*(T)=\Gr_\calT(\sF)=\Gr_\calT(\calO_\sF)$ by the ind-properness, and the inclusion $\calT\subset \calG$ gives a closed immersion $\Gr_\calT\hookto \Gr_\fraka$. 

\begin{dfn} 
Let $\mu \in X_*(T)$. The \emph{global Schubert variety $M_\mu$} is the scheme theoretic closure of the $\calL^+\calG_{\bar{S}}$-orbit of $\tilde{\mu}$ in the BD-Grassmannian $\Gr_{\fraka,\bar{S}}$. 
 \end{dfn}
 
Let us justify the definition. The morphism $\calL^+\calG_{\bar{S}}\to \Gr_{\fraka,\bar{S}}$, $g\mapsto g\tilde{\mu}$ factors by Lemma \ref{actlem} through some smooth affine quotient of $\calL^+\calG_{\bar{S}}$. Then $M_\mu$ is the scheme theoretic closure of this morphism, and hence a reduced flat projective $\bar{S}$-scheme, cf. Theorem \ref{indproj}. Since the fibers of $\calL^+\calG_{\bar{S}}$ are connected, $M_\mu$ has connected and equidimensional fibers. 

\begin{rmk}
Note that the global Schubert variety $M_\mu$ only depends on the $G(\bar{F})$-conjugacy class of $\mu$. In particular, $M_\mu$ is defined over $\calO_E$ where $E$ is the Shimura field of $\mu$, i.e. the finite extension of $F$ defined by the stabilizer in the Galois group of the $G(\sF)$-conjugacy class of $\mu$.
\end{rmk}

In general, the special fiber of $M_\mu$ is not irreducible. It is related to the $\mu$-admissible set (cf. Pappas-Rapoport-Smithling \cite[\S 4.3]{PRS}) as follows, see \eqref{muadmset} below.

For the rest of this section, we assume that $k$ is algebraically closed, i.e. $F=\nF$ and $A=\breve{A}$. Let $B$ be a Borel subgroup containing $A$. Let $R=R(G,A)$ be the set of relative roots, and $R^+=R(B,A)$ the subset of positive roots. Let $W=W(G,A)$ be the Iwahori-Weyl group, cf. \S 1.1, and let $W_0=W_0(G,A)$ be the finite Weyl group. There is a short exact sequence
\[1\longto \La_T\longto W\overset{\pi}{\longto}W_0\longto 1,\] 
where $\pi\colon  W\to W_0$ is the canonical projection and $\La_T=\Gr_\calT(k)$, cf. \cite{HR}. Let $W_\fraka\subset W$ be the subgroup associated with $\fraka$, cf. \S 1.1. Then $\pi|_{W_\fraka}\colon  W_\fraka\to W_0$ is injective, and hence, identifies $W_\fraka$ with a reflection subgroup $W_{0,\fraka}=\pi(W_\fraka)$ of $W_0$. Consider
\[R_\fraka\defined\{\al\in R\;|\; s_\al\in W_{0,\fraka}\},\] 
where $s_\al\in W_0$ denotes the reflection associated with the root $\al$. Then $R_\fraka$ is a root subsystem of $R$, and $R_\fraka^+=R_\fraka\cap R^+$ is a system of positive roots in $R_\fraka$. 

\begin{rmk}
Note that in general $W_{0,\fraka}$ is not a parabolic subgroup of $W_0$, i.e. the root subsystem $R_\fraka\subset R$ is not the system associated with a standard Levi in $G$. In fact, let $W_0'$ be a proper maximal reflection subgroup of $W_0$, i.e. $W_0'$ is a proper maximal subgroup, and is generated by the elements $w\in W_0'$ with $w^2=1$. If $R$ is simple, then there exists a facet $\fraka\subset \scrA$ such that $W_{0,\fraka}=W_0'$, cf. \cite[\S 2, Corollary 1]{DyL}. In particular, all proper maximal root subsystems of $R$ are of the form $R_\fraka$ for some facet $\fraka$.
\end{rmk}

The Kottwitz morphism gives an isomorphism $\La_T\simeq X_*(T)_I$. There is a natural map $X_*(T)_I\to X_*(T)_I\otimes \bbR\simeq X_*(A)_\bbR$, $\bar{\mu}\mapsto\bar{\mu}_\bbR$. Define
\[X_*(T)_I^{\fraka\text{-dom}}\defined\{\bar{\mu}\in X_*(T)_I\;|\;\lan\bar{\mu}_\bbR,\al\ran\geq 0\; \forall \al \in R_\fraka^+ \},\]
where $\lan\str,\str\ran\colon  X_*(A)_\bbR\times X^*(A)_\bbR\to \bbR$ denotes the canonical pairing. Then the canonical map $X_*(T)_I^{\fraka\text{-dom}}\to W_{0,\fraka}\bslash X_*(T)_I$ is bijective. For an element $\bar{\mu}\in X_*(T)_I$, we denote by $\bar{\mu}^{\fraka\text{-dom}}$  the unique representative of $W_{0,\fraka}\cdot\bar{\mu}$ in $X_*(T)_I^{\fraka\text{-dom}}$.

\begin{rmk}
At the one extreme, if $\fraka$ is a special facet, then $W_{0,\fraka}=W_0$, and $X_*(T)_I^{\fraka\text{-dom}}$ is the image of the $B$-dominant elements in $X_*(T)$ under the canonical projection $X_*(T)\to X_*(T)_I$. At the other extreme, if $\fraka$ is an alcove, then $W_{0,\fraka}$ is trivial, and $X_*(T)_I^{\fraka\text{-dom}}=X_*(T)_I$.  
\end{rmk} 

Now fix an alcove $\fraka_C$ which contains $\fraka$ in its closure, and fix a special vertex $\fraka_0$ in the closure of $\fraka_C$. By the choice of $\fraka_0$, we may identify $X_*(A)_\bbR$ with the apartment $\scrA=\scrA(G,A,F)$. Assume that the chamber in $X_*(A)_\bbR$ defined by $B$ is opposite to the chamber which contains $\fraka_C$. This can be arranged by possibly changing $B$. The choice is due to a sign convention in the Kottwitz morphism: The action of $t^\mu\in\La_T$ on $X_*(A)_\bbR$ is given by $v\mapsto v + \mu_\bbR$.

\begin{lem}\label{transmax}
The set ${_\fraka W}^\fraka\cap(W_\fraka \La_T W_\fraka)$ is contained in $\La_T$, and is identified via the Kottwitz morphism with $X_*(T)_I^{\fraka\text{-dom}}$.
\end{lem}
\begin{proof}
Let ${\bar{\mu}}\in X_*(T)_I$. From \cite[Lemma 1.7, Equation (1.10)]{Ri1} one deduces the formula\footnote{Note that the normalization of the Kottwitz morphism in \cite{Ri1} differs by a sign!}
\[l((t^{\bar{\mu}})^\fraka)\;=\; l(t^{\bar{\mu}}) - |\{\al\in R_\fraka^+\;|\; \lan{\bar{\mu}},\al\ran < 0\}|.\]
Note that the root systems $R_\fraka$ and the subsystem of all affine roots $\al$ vanishing on $\fraka$ are elementwise proportional. Hence, $l((t^{\bar{\mu}})^\fraka)$ is maximal if and only if ${\bar{\mu}}={\bar{\mu}}^{\fraka\text{-dom}}$. The uniqueness of the element ${_\fraka(t^{\bar{\mu}})}^\fraka$ implies that
\begin{equation}\label{frakadom}
{_\fraka(t^{\bar{\mu}})}^\fraka\;=\; t^{{\bar{\mu}}^{\fraka\text{-dom}}}
\end{equation}
because both are contained in $(W_\fraka t^{\bar{\mu}} W_\fraka)\cap W^\fraka$, and have the same length. This shows the lemma.
\end{proof}

For $\bar{\mu}\in X_*(T)_I$, and $\rho\in X^*(T)^I$, define the integer
\[\lan\bar{\mu},\rho\ran \defined \lan\mu,\rho\ran,\]
where $\mu$ is a representative of $\bar{\mu}$ in $X_*(T)$. Note that the number $\lan\bar{\mu},\rho\ran\in \bbZ$ does not depend on the choice of $\mu$ by the Galois equivariance of $\lan\str,\str\ran\colon X_*(T)\times X^*(T)\to \bbZ$. For $\bar{\mu}\in X_*(T)_I$, we consider $\bar{\mu}^{\dom}=\bar{\mu}^{\fraka_0\text{-dom}}$.

\begin{cor}\label{lengthcor}
Let $\bar{{{\mu}}}\in X_*(T)_I$, and denote by $t^{\bar{\mu}}$ the associated translation element in $W$. Then
\[l({_\fraka (t^{\bar{\mu}})}^\fraka)\;=\;\lan {\bar{\mu}}^{\dom},2\rho_B\ran,\]
where $2\rho_B$ denotes the sum of the positive absolute roots of $B_\sF$ with respect to $T_\sF$.
\end{cor}
\begin{proof}
By \eqref{frakadom}, we have ${_\fraka (t^{\bar{\mu}})}^\fraka=t^{\bar{\la}}$ with $\bar{\la}\in W_0\cdot {\bar{\mu}}$. The corollary follows from \cite[Lemma 9.1]{Z}.
\end{proof}

Let us recall the definition of the $\mu$-admissible set $\Adm_{\mu}$ for $\mu\in X_*(T)$, cf. \cite[\S 4.3]{PRS}. Let $W_{0}^{\text{abs}}=W_0(G_\sF,B_\sF)$ be the absolute Weyl group. For $\mu\in X_*(T)$ denote by $\tilde{\La}_\mu$ the set of elements $\la\in W_0^{\text{abs}}\cdot \mu$ such that $\la$ is dominant with respect to some $F$-rational Borel subgroup of $G$ containing $T$. Let $\La_\mu$ be the image of $\tilde{\La}_\mu$ under the canonical projection $X_*(T)\to X_*(T)_I$. For $\mu\in X_*(T)$, the \emph{$\mu$-admissible set $\Adm_{\mu}$} is the partially ordered subset of the Iwahori-Weyl group
\begin{equation}\label{muadmset}
\Adm_\mu\defined \{w\in W\;|\;\exists \bar{\la} \in \La_\mu: \;w\leq t^{\bar{\la}}\},
\end{equation}
where $\leq$ is the Bruhat order of $W$. Note that the set $\La_\mu$, and hence $\Adm_\mu$ only depends on the Weyl orbit $W_0^{\text{abs}}\cdot\mu$. Moreover, if $\mu$ is dominant with respect to some $F$-rational Borel subgroup containing $T$, then $\La_\mu=W_0\cdot \bar{\mu}$ where $\bar{\mu}\in X_*(T)_I$ is the image under the canonical projection. 

We define the $\mu$-admissible set $\Adm_\mu^\fraka$ relative to $\fraka$ as the subset of $W$
\[\Adm_\mu^\fraka\defined {_\fraka W}^\fraka\cap (W_\fraka\Adm_\mu W_\fraka).\]
Note that if $\fraka=\fraka_C$ is an alcove, then $\Adm_\mu=\Adm_\mu^\fraka$.

\begin{cor}\label{maxadmcor}
Let $\mu\in X_*(T)$ be $B$-dominant, and denote by $\bar{\mu}$ the image in $X_*(T)_I$. Then the maximal elements in $\Adm_\mu^\fraka$ (wrt $\leq$) are the elements $(W_0\cdot \bar{\mu})^{\fraka\text{-dom}}$. In particular, each maximal element has length $\lan\mu,2\rho_B\ran$, and their number is 
\[|W_{0,\fraka}\bslash W_0/W_{0,\bar{\mu}}|,\] 
where $W_{0,\bar{\mu}}$ is the stabilizer of $\bar{\mu}$ in $W_0$.
\end{cor}
\begin{proof}
The maximal elements in $\Adm_\mu$ are $t^{\bar{\la}}$ where $\bar{\la}\in W_0\cdot\bar{\mu}$. Hence, the maximal elements in $\Adm_\mu^\fraka$ are ${_\fraka (t^{\bar{\la}})}^\fraka$ for $\bar{\la}\in W_0\cdot \bar{\mu}$. By Lemma \ref{transmax}, we have 
\[{_\fraka (t^{\bar{\la}})}^\fraka\;=\;t^{{\bar{\la}}^{\fraka\text{-dom}}},\]
which implies the corollary using Corollary \ref{lengthcor}.
\end{proof}

The combinatorial discussion above allows us to study the irreducible components of the special fiber of $M_\mu$. In fact, the inclusion in Lemma \ref{increlation} below is an equality on reduced loci, cf. \cite{Z} for tamely ramified groups, and the forthcoming manuscript \cite{Ri3} for the general case. Note that this implies Conjecture 4.3.1 of \cite{PRS}, cf. Remark \ref{conjec} below.

\begin{lem}\label{increlation}
Let $\mu\in X_*(T)$ be $B$-dominant, and denote by $\bar{\mu}$ the image in $X_*(T)_I$. The special fiber $M_{\mu,s}$ contains the union of Schubert varieties
\[\bigcup_{w\in\Adm_\mu^\fraka}Y_w.\]
The $Y_{w}$ of maximal dimension, for $w\in\Adm_\mu^\fraka$, are precisely the $Y_{t^{\bar{\la}}}$ with ${\bar{\la}}\in (W_0\cdot \bar{\mu})^{\fraka\text{-dom}}$. Each of them is an irreducible component of $M_{\mu,s}$ of dimension $\lan\mu,2\rho_B\ran$.
\end{lem}
\begin{proof}
The geometric generic fiber $M_{\mu,\oeta}$ is the $L_z^+G_\sF$-orbit closure of $z^\mu\cdot e_0$ in $\Gr_{G,\sF}$, and hence contains the $\sF$-points $z^\la\cdot e_0$ with $\la\in W_0^{\text{abs}}\cdot \mu$. By Lemma \ref{specitorus}, the special fiber $M_{\mu,s}$ contains the $k$-points $t^{\bar{\la}}\cdot e_0$ for $\la\in W_0^{\text{abs}}\cdot \mu$, where $\bar{\la}$ denotes the image in $X_*(T)_I$. The relative Weyl group $W_0$ is contained in the subgroup of $I$-invariant elements in $W_0^{\text{abs}}$, and the canonical projection $X_*(T)\to X_*(T)_I$ is $W_0$-equivariant under this identification. This shows that $M_{\mu,s}$ contains the $k$-points $t^{\bar{\la}}\cdot e_0$ with $\bar{\la}\in W_0\cdot\bar{\mu}$. The $L^+\calG$-invariance of $M_{\mu,s}$ implies that $M_{\mu,s}$ contains the Schubert varieties $Y_{w}$ for $w\in \Adm^\fraka_\mu$. 

The rest of the lemma follows from Corollary \ref{maxadmcor} using that the fibers of $M_\mu$ are equidimensional, cf. Proposition \ref{schubertstr}. 
\end{proof}

\begin{rmk}\label{conjec}
Let us explain how the equality $M_{\mu,s}=\cup_{w\in\Adm_\mu^\fraka}Y_{w}$ on reduced loci implies Conjecture 4.3.1 of \cite{PRS}. Specialize to the case that $\fraka=\fraka_C$ is an alcove, and assume $\mu$ to be $B$-dominant. By the proof of Lemma \ref{increlation}, the special fiber $M_{\mu,s}$ contains all Schubert varieties $Y_{w}$ with $w\leq t^{\bar{\la}}$ for $\la\in W_0^{\text{abs}}\cdot\mu$, where $\bar{\la}$ denotes the image in $X_*(T)_I$. Hence,
\begin{equation}\label{conequal}
\Adm_\mu\;=\;\{w\in W\;|\;\exists \la \in W_0^{\text{abs}}\cdot\mu: \;w\leq t^{\bar{\la}}\}.
\end{equation}
Indeed, $\Adm_\mu$ is clearly contained in the right hand side of \eqref{conequal}, and thus \eqref{conequal} follows from $M_{\mu,s}=\cup_{w\in\Adm_\mu}Y_{w}$. Now Corollary \ref{maxadmcor} shows that the maximal elements in the image of $W_0^{\text{abs}}\cdot\mu$ in $X_*(T)_I$ are precisely the elements $W_0\cdot \bar{\mu}=\La_\mu$. This is Conjecture 4.3.1 of \cite{PRS}.  
\end{rmk}

For $\mu\in X_*(T)$ let $\tau_\mu\colon \calL^+\calG_{\bar{S}}\to M_\mu$, $g\mapsto g\tilde{\mu}$ be the orbit morphism, where $\tilde{\mu}\in M_\mu(\bar{S})$ as above. Let $\mathring{M}_\mu$ be the image of $\tau_\mu$ in the sense of fppf-sheaves. 

\begin{cor}\label{globorbit}
Let $\mu\in X_*(T)$ be dominant with respect to some $F$-rational Borel subgroup containing $T$. Then the fppf-sheaf $\mathring{M}_\mu$ is representable by a smooth open dense subscheme of $M_\mu$.
\end{cor} 
\begin{proof}
Write $\calL^+\calG\simeq \varprojlim_i\calG_i$ as in \eqref{invlim}. The morphism $\tau_\mu$ factors over some $\calG_i$, and $\calG_i/\calG_{i,\mu}\simeq\mathring{M}_\mu$ where $\calG_{i,\mu}\subset \calG_i$ is the stabilizer of $\tilde{\mu}$. We claim that the generic fiber and the special fiber of $\calG_{i,\mu}$ are smooth and geometrically connected of the same dimension. Indeed, $\calG_{i,\mu,\bar{\eta}}$ (resp. $\calG_{i,\mu,s}$) is the $L^+_zG_{\sF}$- (resp. $L^+\calG$-) stabilizer of $z^\mu\cdot e_0$ in $\Gr_{G,\sF}$ (resp. of $t^{\bar{\mu}}\cdot e_0$ in $\Fl_\fraka$), and hence smooth and geometrically connected by Corollary \ref{flagindproper} ii). The quotient $\calG_{i,\bar{\eta}}/\calG_{i,\mu,\bar{\eta}}$ (resp. $\calG_{i,s}/\calG_{i,\mu,s}$) maps isomorphically onto the open cell in the $L^+_zG_{\sF}$- (resp. $L^+\calG$-) orbit closure of $z^\mu\cdot e_0$ in $\Gr_{G,\sF}$ (resp. of $t^{\bar{\mu}}\cdot e_0$ in $\Fl_\fraka$), and Lemma \ref{increlation} implies that $\dim(\calG_{i,\mu,\eta})=\dim(\calG_{i,\mu,s})$. This proves the claim. Now the flat closure of $\calG_{i,\mu,\eta}$ in $\calG_i$ stabilizes $\tilde{\mu}$, and hence, by counting dimensions, is equal to $\calG_{i,\mu}$. This shows that $\calG_{i,\mu}$ is flat and fiberwise smooth, and therefore smooth. Hence, the fppf-quotient $\calG_i/\calG_{i,\mu}$ is representable by a smooth scheme by the main result of \cite{SA}. This gives a quasi-finite separated monomorphism $\tau_\mu\colon \mathring{M}_\mu\to M_\mu$ which is open by Zariski's main theorem. The lemma follows.
\end{proof}

\section{Speciality, parity and monodromy}\label{monpara} 
In \S 3.1 and 3.2, we give a list of characterizations for a facet of being very special (cf. Definition \ref{veryspecialdfn}): geometric (cf. Theorem \ref{monodromy1}), combinatorial (cf. Corollary \ref{abfall} ii)) and arithmetic (cf. Proposition \ref{monodromy2}). This implies Theorem B of the introduction.

Let $F$ be an arbitrary field. Fix a prime $\ell$ different from the characteristic of $F$. Let $\algQl$ be an algebraic closure of the field of $\ell$-adic numbers. For a separated scheme $Y$ of finite type over $F$, we denote by $D_c^b(Y,\algQl)$ the bounded derived category of constructible $\algQl$-complexes. Let $P(Y)$ be the core of the perverse $t$-structure on $D_c^b(Y,\algQl)$ which is an abelian $\algQl$-linear full subcategory of $D_c^b(Y,\algQl)$. If $\calY$ is a ind-scheme separated of finite type over $F$, and $\calY=(Y_\ga)_{\ga\in J}$ an ind-presentation, then let 
\[D_c^b(\calY,\algQl)=\varinjlim_\ga D_c^b(Y_\ga,\algQl)\] 
be the direct limit. Moreover, if $\calY=(Y_\ga)_{\ga\in J}$ is a strict ind-presentation, then let $P(\calY)=\varinjlim_\ga P(Y_\ga)$ be the abelian $\algQl$-linear full subcategory of $D_c^b(\calY,\algQl)$ of perverse sheaves. 

Let $k$ be a either a finite or an algebraically closed field, and specialize to the case that $F=k\rpot{t}$. Let $G$ be a connected reductive group over $F$. Let $\fraka$ be a facet of the extended Bruhat-Tits building $\scrB(G,F)$, and let $\calG=\calG_\fraka$ be the corresponding parahoric group scheme over $\calO_F=k\pot{t}$. Let $\sF$ be the completion of a separable closure of $F$, and denote by $\Ga=\Gal(\sF/F)$ the absolute Galois group. Let $(S,\bar{S},\eta,\oeta,s,\os)$ be the $6$-tuple as above, cf. \S \ref{globSchub}. Let $\Gr_\fraka\to S$ be the BD-Grassmannian associated with the facet $\fraka$, cf. \S \ref{globSchub}. Then $\Gr_\fraka$ is an ind-projective strict ind-scheme, and there is the following cartesian diagram of ind-schemes
\[\begin{tikzpicture}
\matrix(a)[matrix of math nodes, 
row sep=1.5em, column sep=2em, 
text height=1.5ex, text depth=0.45ex] 
{ \Fl_{\fraka}& \Gr_{\fraka} & \Gr_{G}\\ 
s & S & \eta,\\}; 
\path[->](a-1-1) edge node[above] {$i$} (a-1-2) (a-1-3) edge node[above] {$j$} (a-1-2); 
\path[->](a-2-1) edge (a-2-2) (a-2-3) edge (a-2-2); 
\path[->](a-1-1) edge (a-2-1) (a-1-2) edge  (a-2-2) (a-1-3) edge (a-2-3); 
\end{tikzpicture}\]
cf. Corollary \ref{satcat1}. 

Let $\bjay\colon  \Gr_{G,\oeta}\to \Gr_{\fraka,\bar{S}}$ (resp. $\bio\colon  \Fl_{\fraka,\os}\to \Gr_{\fraka,\bar{S}}$) denote the base change of $j$ (resp. $i$). The \emph{functor of nearby cycles $\Psi_\fraka$ associated with $\fraka$} is
\[
\Psi_\fraka\colon D_c^b(\Gr_G,\algQl)\to D_c^b(\Fl_\fraka\times_s\eta,\algQl),\;\;\;\;\;\Psi_\fraka(\calA)=\bio^*\bjay_*(\calA_\oeta).
\] 
Here $D_c^b(\Fl_\fraka\times_s\eta,\algQl)$ denotes the bounded derived category of $\ell$-adic complexes on $\Fl_{\fraka,\os}$ together with a continuous $\Ga$-action compatible with the base $\Fl_{\fraka,\os}$, cf. \cite[\S 5]{GH} and the discussion in the beginning of \cite[\S 9]{PZ}. See [SGA7 II, Expos\'e XIII] for the construction of the Galois action on the nearby cycles. 

The global positive loop group $\calL^+\calG$ acts on $\Gr_\fraka$, and the action factors on each orbit through a smooth affine group scheme which is geometrically connected, cf. Lemma \ref{groupscheme} and \ref{actlem}. Choosing a $\calL^+\calG$-stable ind-presentation of $\Gr_\fraka$, this allows us to consider the category 
\[P_{L_z^+G}(\Gr_G) \;\;\;\;\;\text{(resp. $P_{L^+\calG}(\Fl_\fraka)$)}\]
of $L_z^+G$-equivariant (resp. $L^+\calG$-equivariant) perverse sheaves on $\Gr_G$ (resp. $\Fl_\fraka$) in the generic (resp. special) fiber of $\Gr_\fraka$. Let $P_{L^+\calG}(\Fl_\fraka\times_s\eta)$ be the category of $L^+\calG_\os$-equivariant perverse sheaves on $\Fl_{\fraka,\os}$ compatible with the Galois action, cf. \cite[Definition 9.3]{PZ}. 

\begin{lem}\label{nearlem}
The nearby cycles restrict to a functor $\Psi_\fraka\colon P_{L_z^+G}(\Gr_G)\to P_{L^+\calG}(\Fl_\fraka\times_s\eta)$.
\end{lem}
\begin{proof} The functor $\Psi_\fraka$ preserves perversity by \cite[Appendice, Corollaire 4.2]{I}. An application of the smooth base change theorem to the action morphism $\calL^+\calG\times_S\Gr_\fraka\to\Gr_\fraka$, cf. \eqref{actBD}, implies the equivariance, and the compatibility of the $L^+\calG_\os$-action with the Galois action.
\end{proof}

\subsection{Geometry of special facets} Recall the notion of a special facet (or vertex) in the Bruhat-Tits building $\scrB(G,F)$, cf. \cite{BT1}. Let $A\subset G$ be a maximal $F$-split torus with associated apartment $\scrA=\scrA(G,A,F)$. A facet $\fraka\subset \scrA$ is called \emph{special} if for every affine hyperplane in $\scrA$ there exist a parallel affine hyperplane containing $\fraka$. A facet in the building $\fraka \subset \scrB(G,F)$ is called \emph{special} if $\fraka$ is special in one (hence every) apartment containing $\fraka$.  

For the rest of this subsection, we assume $k$ to be algebraically closed. We consider the functor of nearby cycles $\bar{\Psi}_\fraka$ over $\sF$ 
\[\bar{\Psi}_\fraka\colon  P_{L_z^+G_\sF}(\Gr_{G,\sF})\to P_{L^+\calG}(\Fl_\fraka), \;\;\;\;\;\bar{\Psi}_\fraka(\calA)=\bio^*\bjay_*(\calA).\]
Note that $P_{L_z^+G_\sF}(\Gr_{G,\sF})$ is semi-simple with simple objects the intersection complexes on the $L_z^+G_\sF$-orbit closures, cf. \cite{Ri2}, and hence every object in $P_{L_z^+G_\sF}(\Gr_{G,\sF})$ is defined over some finite extension of $F$. 

The following Theorem proves Theorem B of the introduction.

 \begin{thm}\label{monodromy1} The following properties are equivalent.\smallskip\\
i) The facet $\fraka$ is special.\smallskip\\
ii) The stratification of $\Fl_{\fraka}$ in $L^+\calG$-orbits satisfies the parity property, i.e. in each connected component of $\Fl_\fraka$ all orbits are either even or odd dimensional.\smallskip\\
iii) The category $P_{L^+\calG}(\Fl_\fraka)$ is semi-simple.\smallskip\\
iv) The perverse sheaves $\bar{\Psi}_\fraka(\calA)\in P_{L^+\calG}(\Fl_\fraka)$ are semi-simple for all $\calA\in P_{L_z^+G_\sF}(\Gr_{G,\sF})$.
\end{thm}
 
 Let $A$ be a maximal $F$-split torus, and $T$ its centralizer. Note that $T$ is a maximal torus because $G$ is quasi-split by Steinberg's Theorem. For $\mu\in X_*(T)$, let $M_\mu$ be the corresponding global Schubert variety, cf. \S \ref{globSchub}.

\begin{lem}\label{notirrcor}
Let $\fraka$ be a facet which is not special. Then there exists $\mu\in X_*(T)$ such that the special fiber $M_{\mu,s}$ is not irreducible. 
\end{lem}
\begin{proof} 
Let $\bar{\mu}\in X_*(T)_I$ a strictly dominant element, and let $\mu$ be any preimage in $X_*(T)$ under the canonical projection. By Lemma \ref{increlation}, the special fiber $M_{\mu,s}$ contains at least
\[|W_{0,\fraka}\bslash W_0/W_{0,\bar{\mu}}|\]
irreducible components. This number is $\geq 2$ because $W_{0,\bar{\mu}}$ is trivial, and $W_{0,\fraka}\subset W_0$ is a proper subgroup if $\fraka$ is not special.
\end{proof}

The proof of Theorem \ref{monodromy1} is based on the following geometric lemma.

\begin{lem}\label{irredcomp}
Let $Y$ be a separated scheme of finite type over $k$ which is equidimensional of dimension $d$. Then for the compactly supported intersection cohomology
\[\dim_\algQl\bbH_c^{d}(Y,\IC)\;=\; \#\{\text{irreducible components in}\;\;Y\},\]
where $\IC$ denotes the intersection complex on $Y$.
\end{lem}
\begin{proof}
We may assume that $Y$ is reduced. Let $U\subset Y$ be an open dense smooth subscheme with reduced complement $\iota\colon Z\hookto Y$. Denote by ${^p\!H^*}$ the perverse cohomology functors. There is a cohomological spectral sequence
\begin{equation}\label{cohss}
E^{ij}_2 \defined \bbH^i_c(Z, {^p\!H}^j(\iota^*\IC))\;\Rightarrow\;  \bbH^{i+j}_c(Z,\iota^*\IC).
\end{equation}
Then ${^p\!H^j}(\iota^*\IC) = 0$ for $j> 0$ because $\iota^*$ is $t$-right exact and ${^p\!H^0}(\iota^*\IC) = 0$ by the construction of $\IC$. If $\calA$ is any perverse sheaf on $Z$, then $\bbH^i_c(Z, \calA)=0$ for $i\geq d$, as follows from $\dim(Z)\leq d-1$ and the standard bounds on intersection cohomology. Hence, \eqref{cohss} implies that $\bbH^i_c(Z,\iota^*\IC)=0$ for $i\geq d-1$. The long exact cohomology sequence associated with $U\overset{j}{\hookto} Y \overset{\iota}{\hookleftarrow} Z$ shows
\[\bbH^d_c(U,j^*\IC)\;\overset{\simeq}{\longto}\;\bbH^d_c(Y,\IC).\]
Since $j^*\IC=\algQl[-d]$, this implies the lemma.
\end{proof}

\begin{rmk}
If $Y$ is not necessarily equidimensional, then a refinement of the argument in Lemma \ref{irredcomp} shows that $\dim_\algQl\bbH_c^{d}(Y,\IC)$ is the number of top-dimensional irreducible components, i.e. the irreducible components of dimension $d$. 
\end{rmk}

\begin{proof}[Proof of Theorem \ref{monodromy1}.]
$i)\Rightarrow ii)\Rightarrow iii)$: This is proven in \cite[Lemma 1.1]{RZ}. See also the discussion above [\emph{loc. cit.}], and the displayed dimension formula. Note that the arguments in [\emph{loc. cit.}] do not use the tamely ramified hypothesis. \smallskip\\
$iii)\Rightarrow iv)$: Trivial, since for $\calA\in P_{L_z^+G_\sF}(\Gr_{G,\sF})$, the perverse sheaf $\bar{\Psi}_\fraka(\calA)$ is in $P_{L^+\calG}(\Fl_{\fraka})$, cf. Lemma \ref{nearlem}. \smallskip\\ 
$iv)\Rightarrow i)$: Assume that $\fraka$ is not special. By Lemma \ref{notirrcor}, there exists $\mu \in X_*(T)$ such that the special fiber of the global Schubert variety $M_\mu$ is not irreducible.  Let $\calA$ be the intersection complex on $M_{\mu,\oeta}$. We claim that $\bar{\Psi}_\fraka(\calA)$ is not semi-simple. Assume the contrary. The support of $\bar{\Psi}_\fraka(\calA)$ is equal to the whole special fiber $M_{\mu,s}$ by \cite[Lemma 7.1]{Z} and, since $\bar{\Psi}_\fraka(\calA)$ is $L^+\calG$-equivariant, the intersection complex on $M_{\mu,s}$ must be a direct summand of $\bar{\Psi}_\fraka(\calA)$. Let $d=\dim(M_{\mu,\oeta})=\dim(M_{\mu,s})$. Taking cohomology 
\[\bbH^d(M_{\mu,\oeta},\calA)\simeq \bbH^d(M_{\mu, s},\bar{\Psi}_\fraka(\calA))\]
contradicts Lemma \ref{irredcomp} because the left side is $1$-dimensional, and the right side is at least $2$-dimensional. This shows that $\bar{\Psi}_\fraka(\calA)$ is not semi-simple.
\end{proof}

As a consequence of the proof, we obtain the following corollary which implies items iv) and v) of Theorem B of the introduction.

\begin{cor}\label{abfall} The following properties are equivalent to properties i)-iv) of Theorem \ref{monodromy1}.\smallskip\\
v) The special fiber of the global Schubert $M_\mu$ in $\Gr_\fraka$ is irreducible for all $\mu\in X_*(T)$.\smallskip\\
vi) The admissible set $\Adm_\mu^\fraka$ has a unique maximal element for all $\mu\in X_*(T)_I$.
\end{cor}
\vspace{-0.4cm}
\hfill\ensuremath{\Box} 

\subsection{Arithmetic of very special facets} In this subsection $k$ is finite, so that $F=k\rpot{t}$ is a local non-archimedean field. We will show that the property of a facet of being very special (cf. Definition \ref{veryspecialdfn} below) is related to the vanishing of the monodromy operator on Gaitsgory's nearby cycles functor, and hence to the triviality of the weight filtration. 

Let $\nF$ be the completion of the maximal unramified subextension of $\sF$, and let $\sig\in \Gal(\nF/F)$ be the Frobenius. Note that there is a $\sigma$-equivariant embedding of buildings
\[\iota\colon \scrB(G,F)\;\longto\;\scrB(G,\nF),\] 
which identifies $\scrB(G,F)$ with the $\sigma$-fixpoints in $\scrB(G,\nF)$. In \cite{RZ}, Zhu defines the notion of very special facets as follows.

\begin{dfn}\label{veryspecialdfn}
A facet $\fraka\subset\scrB(G,F)$ is called \emph{very special} if the unique facet $\nfraka\subset\scrB(G,\nF)$ with $\iota(\fraka)\subset \nfraka$ is special.
\end{dfn}

\begin{rmk}
Every hyperspecial facet is very special. By \cite{Ti} all hyperspecial facets are conjugate under the adjoint group, whereas this is \emph{not} true for very special facets. In fact, the only case among all absolutely simple groups (up to central isogeny), where this is not true, is a ramified unitary group in odd dimensions, cf. [\emph{loc. cit.}].
\end{rmk}

\begin{lem}\label{quasisplitlem} i) If $\fraka$ is a very special facet, then $\fraka$ is special.\smallskip\\
ii) The building $\scrB(G,F)$ contains very special facets if and only if the group $G$ is quasi-split.
\end{lem}
\begin{proof}
Part i) follows from \cite[1.10.1]{Ti}, and part ii) from \cite[Lemma 6.1]{RZ}.
\end{proof}

Recall the construction of the monodromy operator, see \cite[\S 5]{GH} for details. Let $I\subset \Ga$ be the inertia subgroup, i.e. $\Ga/I=\Gal(\nF/F)$. Let $P\subset I$ be the wild inertia group, so that
\[I/P\;=\;\prod_{\ell'\not= p}\bbZ_{\ell'}(1),\]
and denote by $t_\ell\colon I\to \bbZ_\ell(1)$ the composition of $I\to I/P$ with the projection on $\bbZ_\ell(1)$. If $Y$ is a separated $k$-scheme of finite type, then let $D_c^b(Y\times_s\eta,\algQl)$ be the bounded derived category of constructible $\algQl$-complexes together with a continuous $\Ga$-action as above. Let $\calA\in D_c^b(Y\times_s\eta,\algQl)$, and denote by $\rho\colon  I\to \Aut_{D_c^b}(\calA)$ the inertia action. Then $\rho(I)$ acts quasi-unipotently in the sense that there is an open subgroup $I_1\subset I$ such that $\rho(g)-\id_\calA$ acts nilpotently for all $g\in I_1$. There is a unique nilpotent morphism
\[N_\calA\colon \calA(1)\;\longto\;\calA\]
characterized by the equality $\rho(g)=\exp(t_\ell(g)N_\calA)$ for all $g\in I_1$, and $N_\calA$ is independent of $I_1$. 

The choice of a Frobenius element in $\Ga$ defines a semi-direct product decomposition $\Ga=I\rtimes \Gal(\bar{k}/k)$. Recall that if $\calA\in P(Y\times_s\eta)$ then, by restricting the $\Ga$-action on $\calA$ to $\Gal(\bar{k}/k)$, the underlying perverse sheaf is equipped with a continuous $\Gal(\bar{k}/k)$-descent datum, and hence defines an element $\calA_0\in P(Y)$. Then $\calA$ is called \emph{mixed} (resp. \emph{pure of weight $w$}) if $\calA_0$ is mixed (resp. pure of weight $w$). Note that all Frobenius elements are conjugate under the inertia group $I$, and hence the notion of mixedness (resp. purity) does not depend on this choice, cf. [Weil2].

Let $\om$ be the global cohomology functor with Tate twists included
\begin{equation}\label{fiberfun1}
\om(\str)\defined\bigoplus_{i\in\bbZ}(R^i\Ga(\Gr_{G,\sF},(\str)_\sF)({i\over 2}))\colon \;P_{L_z^+G}(\Gr_G)\;\longto\;\vs_\algQl.
\end{equation}
Note that if $\calA$ is an intersection complex on a $L_z^+G$-stable closed subscheme of $\Gr_G$, then the Galois action on $\om(\calA)$ factors through a finite quotient of $\Ga$, cf. \cite[Appendix]{RZ}. This explains the Tate twist in \eqref{fiberfun1}.

The following proposition together with Theorem \ref{monodromy1} implies item vi) of Theorem B of the introduction.

\begin{prop}\label{monodromy2} Let $\calA\in P_{L_z^+G}(\Gr_G)$ such that the $\Ga$-action on $\om(\calA)$ factors through a finite quotient. Then the following properties are equivalent.\smallskip\\
i) The perverse sheaf $\Psi_\fraka(\calA)_\os\in P_{L^+\calG_\os}(\Fl_{\fraka,\os})$ is semi-simple.  \smallskip\\
ii) The nearby cycles complex $\Psi_\fraka(\calA)$ is pure of weight $0$. \smallskip\\
iii) The monodromy operator $N_{\Psi_\fraka(\calA)}=0$ vanishes.
\end{prop}

This proposition and Theorem \ref{monodromy1} imply that the monodromy of $\Psi_\fraka$ is non-trivial whenever $\fraka$ is not very special. Note that in Theorem \ref{monodromy1} the residue field is assumed to be algebraically closed, and hence the notion of special facets and very special facets coincide. In fact, one can show that the monodromy of $\Psi_\fraka$ is maximally non-trivial, cf. \cite{Ri3}. The equivalence $ii)\Leftrightarrow iii)$ is a special case of the weight monodromy conjecture for perverse sheaves proven by Gabber \cite{BB}. Since the proof is easy using semi-continuity of weights, we explain it below.

\begin{lem}\label{Gabberlem}
Let $\fraka$ be a facet, and let $\calA\in P_{L_z^+G}(\Gr_G)$. Then $\Psi_\fraka(\calA)$ is pure if and only if $N_{\Psi_\fraka(\calA)}=0$. In this case, $\Psi_\fraka(\calA)$ is pure of weight $0$.
\end{lem}
\begin{proof}
If $\Psi_\fraka(\calA)$ is pure, then $N_{\Psi_\fraka(\calA)}\colon \Psi_\fraka(\calA)(1)\to\Psi_\fraka(\calA)$ vanishes due to weight reasons. Conversely suppose that $N_{\Psi_\fraka(\calA)}=0$. By \cite[Proof of Corollaire 4.6 (c) Cas (ii)]{I}, there is a distinguished triangle 
\[\bar{i}^*j_*\calA \longto \Psi_\fraka(\calA)\overset{0}{\longto} \Psi_\fraka(\calA) \longto\]
where $j\colon \Gr_{\fraka,\eta}\to \Gr_\fraka$ denotes the open embedding. Hence, on perverse cohomology 
\[\Psi_\fraka(\calA)\;\simeq\; {^pH}^{0}(\bar{i}^*j_*\calA)\;\simeq\; \bar{i}^*j_{!*}(\calA)[-1].\] 
This implies for the weights
\[w(\Psi_\fraka(\calA))\leq w(j_{!*}(\calA))\leq w(\calA)=0,\]
and since $\Psi_\fraka$ commutes with duality, we get $w(\Psi_\fraka(\calA))=0$.
\end{proof}

\begin{proof}[Proof of Proposition \ref{monodromy2}.] The implication $ii)\Rightarrow i)$ is a consequence of Gabber's Decomposition Theorem (cf. \cite[Chapter III.10]{KW}) because $\Psi_\fraka(\calA)$ is defined over the ground field $k$. In view of Lemma \ref{Gabberlem}, we are reduced to proving the implication $i) \Rightarrow ii)$: Let $\calA\in P_{L_z^+G}(\Gr_G)$ such that the $\Ga$-action on $\om(\calA)$ factors through a finite quotient. Hence, after a finite base change $S'\to S$, we may assume that the Galois action on the global cohomology $\om(\calA)$ is trivial. By Deligne [Weil2], the nearby cycles $\Psi_\fraka(\calA)$ are mixed because $\Gr_{\fraka}$ is already defined over a smooth curve over $k$. Let  
\[\text{gr}^\bullet\Psi_\fraka(\calA)\;\simeq\;\bigoplus_\beta\calA_\beta,\]
be the associated graded of the weight filtration, where $\calA_\beta\in P_{L^+\calG}(\Fl_{\fraka})$ is pure of weight $\beta$. Let $\om_s\colon P_{L^+\calG}(\Fl_{\fraka}\times_s\eta)\to \vs_{\algQl}$ be the global cohomology with Tate twists included as in \eqref{fiberfun1}. If $\Psi_\fraka(\calA)_\os$ is semi-simple, then $\om_s(\Psi_\fraka(\calA))=\om_s(\text{gr}^\bullet\Psi_\fraka(\calA))$ as Galois representations. Because the Galois action on $\om_s(\Psi_\fraka(\calA))\simeq \om(\calA)$ is trivial, it follows that $\om_s(\calA_\beta)=0$ for $\beta\not =0$. But $\calA_{\beta,\os}$ is the direct sum of intersection complexes, and hence $\om_s(\calA_\beta)=0$ implies $\calA_\beta=0$, cf. Lemma \ref{irredcomp}. This shows that $\Psi_\fraka(\calA)$ is pure of weight $0$.
\end{proof}

\section{Satake categories}\label{satcatpara}
In \S \ref{unramsatcatpara}, we recall some facts from the unramified geometric Satake equivalence, cf. \cite{Gi}, \cite{MV} for complex coefficients, and \cite{Ri2}, \cite[Appendix]{RZ} for the case of $\ell$-adic coefficients. In \S \ref{ramsatcatpara}, the ramified geometric Satake equivalence for ramified groups of Zhu \cite{RZ} is explained. Zhu considers in \cite{RZ} tamely ramified groups. We extend his results to include the wildly ramified case. The proof of Theorem C from the introduction is given at the end of \S \ref{ramsatcatpara}.

\subsection{The unramified Satake category}\label{unramsatcatpara} Let $G$ be a connected reductive group over any field $F$. Let $\Gr_G$ be the affine Grassmannian over $\Gr_G$ with its left action by the positive loop group $L_z^+G$, cf. \S \ref{globSchub}. Let $\sF$ be a separable closure of $F$, and denote by $\Ga$ the absolute Galois group. Let $J$ be the set of Galois orbits on the set of $L_z^+G_{\sF}$-orbits in $\Gr_{G,\sF}$. Each $\ga\in J$ defines a connected smooth $L_z^+G$-invariant subscheme $\calO_\ga$ over $F$. We have a $L^+_zG$-invariant ind-presentation of the reduced locus $(\Gr_G)_\red=\varinjlim_\ga \olO_\ga$  by the reduced closures $\olO_\ga$. 

Fix a prime $\ell$ different from the characteristic of $F$. Let 
\[P_{L_z^+G}(\Gr_G)=\varinjlim_\ga P_{L_z^+G}(\olO_\ga)\] 
be the category of $L^+_zG$-equivariant $\ell$-adic perverse sheaves on $\Gr_G$, cf. \S \ref{monpara}.

\begin{lem}\label{gen1}
The category $P_{L_z^+G}(\Gr_G)$ is abelian $\algQl$-linear, and its simple objects are middle perverse extensions $i_*j_{!*}(V[\dim(\calO_\ga)])$, where  $j\colon \calO_\ga\hookto\olO_\ga$, $i\colon \olO_\ga\hookto\Gr_G$, and $V$ is a simple $\ell$-adic local system on $\Spec(F)$.
\end{lem}
\begin{proof}
By \cite{LO}, the simple objects in $P(\Gr_G)$ are of the form $\calA=i_*j_{!*}(\calA_0)$ for $j\colon U\to \bar{U}$ a smooth irreducible open subscheme of a closed subscheme $i\colon \bar{U}\to\Gr_G$, and $\calA_0[-\dim(U)]$ a simple $\ell$-adic local system on $U$. If $\calA$ is $L_z^+G$-equivariant, then $U$ is $L_z^+G$-invariant. In this case, $U_{\sF}$ is a single Galois orbit of $L_z^+G_{\sF}$-orbits, and hence $U=\calO_\ga$ for some $\ga\in J$. On the other hand, the stabilizers of the $L_z^+G_{\sF}$-action are connected by \cite[Lemme 2.3]{NP}, and thus $\calA_0=V[\dim(U)]$ where $V$ is a simple $\ell$-adic local system on $\Spec(F)$. 
\end{proof}

If $F$ is separably closed, the category $P_{L_z^+G}(\Gr_G)$ is semi-simple with simple objects the intersection complexes on the $L_z^+G$-orbit closures, cf. \cite{Ri2}.

\begin{dfn}
The \emph{unramified Satake category $\Sat_{G,\sF}$ over $\sF$} is the category $P_{L_z^+G_\sF}(\Gr_{G,\sF})$.
\end{dfn}

A version of $\Sat_{G,\sF}$ over the ground  field $F$ is defined as follows. Fix $\sqrt{p}\in\algQl$ so that half-integral Tate twists are defined. For a complex $\calA\in D_c^b(Y,\algQl)$ on any separated scheme $Y$ of finite type over $F$, we introduce the shifted and twisted version $\calA\lan m\ran=\calA[m]({m\over 2})$ for $m\in\bbZ$. Now let $Y$ be a equidimensional smooth scheme over $F$. Let $F'/F$ be a finite separable field extension. Then we say that a complex $\calA_0$ in $D_c^b(Y,\algQl)$ is \emph{constant on $Y$ over $F'$} if $\calA_{0, F'}$ is a direct sum of copies of $\algQl\lan \dim(Y)\ran$. 

For every $\ga\in J$, let $\iota_\ga\colon \calO_\ga\to\Gr_G$ be the corresponding locally closed embedding.

\begin{dfn}
The \emph{unramified Satake category $\Sat_G$ over $F$} is the full subcategory of $P_{L_z^+G}(\Gr_G)$ of semi-simple objects $\calA$ such that there exists a finite separable extension $F'/F$ with the property that the $0$-th perverse cohomology $^p\!H^0(\iota_\ga^*\calA)$ and $^p\!H^0(\iota_\ga^!\calA)$ are constant on $\calO_\ga$ over $F'$ for each $\ga\in J$.
\end{dfn}

For any $\ga\in J$, we define $\IC_\ga= i_*j_{!*}(\algQl\lan\dim(\calO_\ga)\ran)$ where $\calO_\ga\overset{j}{\hookto}\olO_\ga\overset{i}{\hookto}\Gr_G$ is the open embedding into the closure.

\begin{lem} \label{simpobj}
Let $\calA\in P_{L_z^+G}(\Gr_G)$ be a simple object. Then $\calA\in \Sat_G$ if and only if there is an $\ga\in J$ such that $\calA\simeq\IC_\ga\otimes V$ where $V$ is a simple local system on $\Spec(F)$ that is trivial over some finite extension $F'/F$.
\end{lem}
\begin{proof}
Let $\calA=i_*j_{!*}(V[\dim(\calO_\ga)])$ be simple for some $\ga\in J$. Assume that $\calA\in\Sat_G$. Then there exists $F'/F$ finite such that $^p\!H^0(\iota^*\calA)=V[\dim(\calO_\ga)]$ is constant over $F'$ for $\iota\colon \calO_\ga\to\Gr_G$, i.e. $V[\dim(\calO_\ga)]=\algQl\lan \dim(\calO_\ga)\ran\otimes V$ where $V$ is a local system that is trivial over $F'$. Since the middle perverse extension commutes with smooth morphisms, we obtain $\calA\simeq\IC_\ga\otimes V$. The converse follows from the fact that $^p\!H^0(\iota^o_{\ga'}\calA)=0$, unless $\ga'=\ga$ and in this case $^p\!H^0(\iota_\ga^o\calA)=V_0[\dim(\calO_\ga)]$ for both restrictions $\iota_\ga^o=\iota_\ga^*$ and $\iota_\ga^o=\iota_\ga^!$.
\end{proof}

We recall from \cite{Ri2} that the category $P_{L_z^+G}(\Gr_G)$ is equipped with a symmetric monoidal structure with respect to the convolution product $\star$ uniquely determined by the property that the global cohomology functor $\om\colon  P_{L_z^+G}(\Gr_G)\to \vs_\algQl$ is symmetric monoidal, cf. \eqref{fiberfun1} for the definition of $\om$.

Recall the classical geometric Satake isomorphism, first over $\sF$. The tuple $(\Sat_{G,\sF},\star)$ is a neutralized Tannakian category with fiber functor $\om_\sF$, and the group of tensor automorphisms $\hat{G}=\Aut^{\star}(\om_\sF)$ is a connected reductive group over $\algQl$ whose root datum is dual to the root datum of $G_\sF$ in the sense of Langlands. 

Now for arbitrary $F$, it is shown in \cite[Appendix]{RZ} that for any object $\calA\in\Sat_G$ the $\Ga$-action on $\om(\calA)$ factors through a finite quotient, cf. the Tate twist in \eqref{fiberfun1}. Hence, $\Ga$ acts on $\hat{G}$ via a finite quotient, and we may form $\LG=\hat{G}\rtimes \Ga$ considered as a pro-algebraic group over $\algQl$ with neutral component $\hat{G}$. In this way, for every $\calA\in \Sat_G$, the cohomology $\om(\calA)$ is an algebraic representation of the affine group scheme $\LG$. Denote by $\Rep_{\algQl}(\LG)$ (resp. $\Rep_{\algQl}(\hat{G})$) the tensor category of algebraic representations of $\LG$ (resp. $\hat{G}$) over $\algQl$.
 
\begin{thm}[\text{[34, Appendix]}] \label{GeoSat1} i) The category $\Sat_G$ is stable under the convolution product, and $(\Sat_G,\star)$ is a semi-simple abelian tensor subcategory of $(P_{L_z^+G}(\Gr_G),\star)$.\smallskip\\
ii) The base change to $\sF$ defines a tensor functor $(\str)_\sF\colon  (\Sat_G,\star)\longto (\Sat_{G,\sF},\star)$, and the following diagram of functors between abelian tensor categories 
\[
\begin{tikzpicture}[baseline=(current  bounding  box.center)]
\matrix(a)[matrix of math nodes, 
row sep=1.5em, column sep=2em, 
text height=1.5ex, text depth=0.45ex] 
{(\Sat_G,\star)&(\Sat_{G,\sF},\star) \\ 
(\Rep_{\algQl}(\LG),\otimes)&(\Rep_{\algQl}(\hat{G}),\otimes) \\}; 
\path[->](a-1-1) edge node[above] {$(\str)_{\sF}$} (a-1-2);
\path[->](a-2-1) edge node[above] {$\res$} (a-2-2);
\path[->](a-1-1) edge node[right] {$\om$} (a-2-1);
\path[->](a-1-2) edge node[right] {$\om_{\sF}$} (a-2-2);
\end{tikzpicture}
\]
is commutative up to natural isomorphism, where $\res$ denotes the restriction of representations along $\hat{G}\hookto \LG$.
\end{thm}

\begin{cor}\label{satcor} Let $\calA\in \Sat_G$. Then the Galois group acts trivially on $\om(\calA)$ if and only if $\calA$ is a direct sum of $\IC_\ga$ for $\ga\in J$ such that $\calO_{\ga,\sF}$ is connected. 
\end{cor}
\begin{proof}
We may assume that $\calA$ is simple, and hence $\calA=\IC_\ga\otimes V$ for some $\ga\in J$ and some local system $V$ on $\Spec(F)$ by Lemma \ref{simpobj}. If $\Ga$ acts trivially on $\om(\calA)$, then $V$ is trivial, and $\calO_{\ga,\sF}$ is connected. Conversely, if $\calA=\IC_\ga$ for $\ga\in J$ with $\calO_{\ga,\sF}$ connected, then $\Ga$ acts trivial on $\om(\calA)$ by \cite[Appendix]{RZ}, cf. the Tate twist in \eqref{fiberfun1}.
\end{proof}

\begin{rmk}\label{pinrem}
i) For an interpretation of the whole abelian tensor category $(P_{L_z^+G}(\Gr_{G}),\star)$ in terms of the dual group see \cite[\S 5]{Ri2}.\smallskip\\
ii) The group of tensor automorphisms $\hat{G}$ admits a canonical pinning $(\hat{G},\hat{B},\hat{T},\hat{X})$, cf. Appendix \ref{fixpointapp} for the definition of a pinning. Moreover, the action of $\Ga$ on $\hat{G}$ is via pinned automorphisms. As explained in \S 4 of \cite{RZ}, the canonical pinning is constructed as follows. The cohomological grading on $\om$ defines a one parameter subgroup $\bbG_m\to \hat{G}$, and the centralizer $\hat{T}$ is a maximal torus. Let $\calL$ be an ample line bundle on $\Gr_G$. Then its isomorphism class $[\calL]\in \Pic(\Gr_G)$ is unique. Cup product with the first Chern class $c_1([\calL])\in H^2(\Gr_{G,\sF},\bbZ_\ell(1))$ defines a principal nilpotent element $\hat{X}\in \Lie(\hat{G})$. This in turn determines the Borel subgroup $\hat{B}$ with $\hat{T}\subset \hat{B}$ and $\hat{X}\in \Lie(\hat{B})$ uniquely. Since the Galois group $\Ga$ fixes the cohomological grading and $[\calL]$, it acts on $\hat{G}$ via pinned automorphisms.
\end{rmk}

\subsection{The ramified Satake category}\label{ramsatcatpara} Let $k$ be a finite field, and let $G$ be a connected reductive group over the Laurent power series field $F=k\rpot{t}$. Let $\fraka$ be a facet in the Bruhat-Tits building $\scrB(G,F)$, and denote by $\calG=\calG_\fraka$ the associated parahoric group scheme over $\calO_F$.

There is the convolution product, cf. \cite{Ga}, \cite{PZ}
\[
\str\star\str\colon P(\Fl_\fraka)\times P_{L^+\calG}(\Fl_\fraka)\longto D_c^b(\Fl_\fraka,\algQl).
\]
Note that $P_{L^+\calG}(\Fl_\fraka)$ is not stable under $\star$ in general, i.e. the convolution of two perverse sheaves need not to be perverse again. For the preservation of perversity we need a hypothesis on $\fraka$. 

For the rest of the section, let $\fraka$ be a very special facet, cf. Definition \ref{veryspecialdfn}.

\begin{dfn}
The \emph{ramified Satake category $\Sat_{\fraka,\os}$ over $\os$} is the category $P_{L^+\calG_\os}(\Fl_{\fraka,\os})$.
\end{dfn}

\begin{rmk}
The connection with \S 2.1 is as follows. The choice of a hyperspecial facet $\fraka$ is equivalent to the choice of a Chevalley model of $G$ over $\calO_F$. In this case, the BD-Grassmannian $\Gr_\fraka$ is constant over $S$, and the nearby cycles functor $\bar{\Psi}_{\fraka}\colon \Sat_{G,\oeta}\to \Sat_{\fraka,\os}$ is an equivalence of tensor categories, cf. the proof of Theorem \ref{Zhu1} below.
\end{rmk}

A version of $\Sat_{\fraka,\os}$ with Galois action is defined as follows. For a finite intermediate extension $F\subset F'\subset \sF$, let $(S',\bar{S},\eta',\oeta,s',\os)$ be the associated $6$-tuple with Galois group $\Ga'=\Gal(\sF/F')$. Then there is the functor
\[\res_{F'/F}\colon  P_{L^+\calG}(\Fl_\fraka\times_s\eta)\;\longto\; P_{L^+\calG}(\Fl_\fraka\times_{s'}\eta')\]
given by restricting the Galois action from $\Ga$ to the subgroup $\Ga'$. Furthermore, there is the functor
\[(\str)_{\os}\colon  P_{L^+\calG}(\Fl_\fraka)\;\longto\; P_{L^+\calG}(\Fl_\fraka\times_s\eta)\]
given by pullback along $\Fl_{\fraka,\os}\to \Fl_\fraka$. Note that $(\str)_\os$ is fully faithful with essential image consisting of the objects $\calA\in P_{L^+\calG}(\Fl_\fraka\times_s\eta)$ such that the inertia acts trivially.

\begin{dfn} 
The \emph{ramified Satake category $\Sat_\fraka$ over $s$} is the full subcategory of objects $\calA\in P_{L^+\calG}(\Fl_\fraka\times_s\eta)$ with the property that there exists a finite separable extension $F'/F$ such that \smallskip\\
\phantom{hallo}a) the inertia $I'\subset \Ga'$ acts trivially on $\res_{F'/F}(\calA)$, and \smallskip\\
\phantom{hallo}b) the perverse sheaf $\res_{F'/F}(\calA)\in P_{L^+\calG}(\Fl_\fraka)$ is semi-simple and pure of weight $0$.
\end{dfn}

We denote by $\om_s\colon  \Sat_\fraka\to\vs_\algQl$ the global cohomology, with Tate twists included, as in \eqref{fiberfun1}, and likewise $\om_\os\colon  \Sat_{\fraka,\os}\to\vs_\algQl$. Since the Galois group $\Ga$ acts via a finite quotient on $\hat{G}=\Aut^\star(\om)$, we may consider the invariants $\hat{G}^I$ under the inertia group. Then $\hat{G}^I\subset \hat{G}$ is a reductive subgroup which is not connected in general. The group $\Ga$ operates on $\hat{G}^I$, and we form the semi-direct product $\LG_{\on{r}}=\hat{G}^I\rtimes\Ga$, considered as a pro-algebraic group over $\algQl$. Hence, $\LG_{\on{r}}\hookto {\LG}$ is a closed subgroup scheme. 

Recall that there is the nearby cycles functor $\Psi_\fraka\colon P_{L_z^+G}(\Gr_G)\to P_{L^+\calG}(\Fl_\fraka)$ associated with $\fraka$, cf. \S \ref{monpara}.

\begin{thm}\label{Zhu1}
Let $\fraka$ be very special. \smallskip\\
i) The category $\Sat_\fraka$ is semi-simple and stable under the convolution product $\star$.\smallskip\\
ii) If $\calA\in\Sat_G$, then $\Psi_\fraka(\calA)\in \Sat_\fraka$, and the pair $(\Sat_\fraka,\star)$ admits a unique structure of a symmetric monoidal category such that $\Psi_\fraka\colon  (\Sat_G,\star)\to (\Sat_\fraka,\star)$ is symmetric monoidal. \smallskip\\
iii) The following diagram of functors of abelian tensor categories 
\[
\begin{tikzpicture}[baseline=(current  bounding  box.center)]
\matrix(a)[matrix of math nodes, 
row sep=1.5em, column sep=2em, 
text height=1.5ex, text depth=0.45ex] 
{(\Sat_G,\star)&(\Sat_\fraka,\star) \\ 
(\Rep_{\algQl}(\LG),\otimes)&(\Rep_{\algQl}({\LG}_{\on{r}}),\otimes) \\}; 
\path[->](a-1-1) edge node[above] {$\Psi_\fraka$} (a-1-2);
\path[->](a-2-1) edge node[above] {$\res$} (a-2-2);
\path[->](a-1-1) edge node[right] {$\om$} (a-2-1);
\path[->](a-1-2) edge node[right] {$\om_s$} (a-2-2);
\end{tikzpicture}
\]
is commutative up to natural isomorphisms, and the vertical arrows are equivalences.
\end{thm} 
\begin{proof}
We explain the modifications in Zhu's proof of Theorem \ref{Zhu1}.\smallskip\\
\emph{The geometric equivalence:} Let $\bar{S}=\Spec(\calO_\sF)$, and consider the base change $\Gr_{\fraka,\bar{S}}=\Gr_\fraka\times_S\bar{S}$. Let $\Sat_{G,\oeta}=P_{L_z^+G_\oeta}(\Gr_{G,\oeta})$ (resp. $\Sat_{\fraka,\os}=P_{L^+\calG_\os}(\Fl_{\fraka,\os})$) be the Satake category over $\oeta$ (resp. $\os$). Recall that there is the nearby cycles functor, cf. \S \ref{monpara}
\[\bar{\Psi}_\fraka\colon  \Sat_{G,\oeta}\;\longto\; \Sat_{\fraka,\os}.\] 
We go through the arguments in Zhu's paper \cite{RZ}. \smallskip\\
a) The category $\Sat_{\fraka,\os}$ is semi-simple and stable under the convolution product $\star$. In particular, the pair $(\Sat_{\fraka,\os},\star)$ is a monoidal category.\smallskip\\
The category $P_{L^+\calG_\os}(\Fl_{\fraka,\os})$ is semi-simple by Theorem \ref{monodromy1} iii) (Lemma 1.1 in [\emph{loc. cit.}]). We show that it is stable under convolution. Let $\calA\in \Sat_{G,\oeta}$. The monodromy of $\bar{\Psi}_\fraka(\calA)$ is trivial by Proposition \ref{monodromy2} iii) (Lemma 2.3 in [\emph{loc. cit.}]). As in the proof of Lemma \ref{Gabberlem} this implies the formula
\begin{equation}\label{important}
\bar{\Psi}_\fraka(\calA)\;\simeq\;\bar{i}^*j_{!*}(\calA),
\end{equation}
which is Corollary 2.5 of [\emph{loc. cit.}]. Hence, \eqref{important} holds for all $\calA\in\Sat_{G,\oeta}$. Let $\mu\in X_*(T)$ be dominant with respect to some $F$-rational Borel subgroup of $G$. Note that $G$ is quasi-split by Lemma \ref{quasisplitlem}. Let $M_\mu$ be the global Schubert variety, and let $\IC_\mu$ be the intersection complex on $M_{\mu,\oeta}$. Then the intersection complex $\IC_{t^{\bar{\mu}}}$ on the Schubert variety $Y_{t^{\bar{\mu}}}$ in the special fiber appears with multiplicity $1$ in $\bar{\Psi}_\fraka(\IC_\mu)$ (Lemma 2.6 in [\emph{loc. cit.}]). This follows from the compatibility of nearby cycles along smooth morphisms applied to the open immersion $\mathring{M}_\mu\hookto M_\mu$, cf. Corollary \ref{globorbit} and Lemma \ref{increlation}. We claim that Proposition 2.7 [\emph{loc. cit.}] holds, i.e. for $\calA\in \Sat_{G,\oeta}$ and $\calB\in \Sat_{\fraka,\os}$, there is a canonical isomorphism $\bar{\Psi}_\fraka(\calA)\star\calB\simeq\calB\star\bar{\Psi}_\fraka(\calA)$ and both are objects in $\Sat_{\fraka,\os}$. The Zhu's proof carries over word by word in replacing $(\bbA^1_k,0)$ by the pointed curve $(X,x)$ with $\calO_x=\calO_F$, cf. \S \ref{globSchub}: Let $\Gr^{\rm BD}_X$ be the contravariant functor on the category of $k$-schemes parametrizing isomorphism classes of triples $(y,\calF,\al)$ with
\[
\begin{cases}
y\colon T\to X\; \text{is a morphism of $k$-schemes};\\
\calF \;\text{a right $\calG_T$-torsor on $X_T$};\\
\al\colon \calF_{X_T\bslash\{\Ga_x\cup \Ga_y\}}\overset{\simeq}{\longto}\calF^0|_{X_T\bslash\{\Ga_x\cup\Ga_y\}}\;\text{a trivialization},
\end{cases}
\]
where $\calF^0$ is the trivial torsor. There is the projection $\Gr_X^{\rm BD}\to X$ and we denote 
\[\Gr^{\rm BD}\defined \Gr^{\rm BD}\times_X S,\]
which is representable by an ind-projective strict ind-scheme. Moreover, on fibers
\[\Gr^{\rm BD}_\oeta \simeq \Gr_{G,\oeta}\times_{\os} \Fl_{\fraka,\os} \;\;\;\;\;\;\text{and}\;\;\;\;\;\;\; \Gr^{\rm BD}_s\simeq \Fl_{\fraka,\os}.\]
Denote $\Psi^{\rm BD}: P(\Gr_{G,\oeta})\to P(\Fl_{\fraka,\os})$ the nearby cycles functor. Consider the perverse sheaf $\calA\boxtimes \calB$ on $\Gr_\oeta^{\rm BD}$. Then $\Psi^{\rm BD}(\calA\boxtimes \calB)$ is $L^+\calG_{\os}$-equivariant, i.e. an object in $\Sat_{\fraka, \os}$ and hence semi-simple. Thus the monodromy vanishes by the Proof of Proposition \ref{monodromy2}, and as in the Proof of \cite[Theorem 7.3]{Z} one shows that there are canonical isomorphisms
\[\bar{\Psi}_\fraka(\calA)\star\calB\;\simeq\; \Psi^{\rm BD}(\calA\boxtimes \calB) \;\simeq\; \bar{i}^*j_{!*}(\calA\boxtimes \calB).\]
Finally, the $\bbZ/2$-action on $\Gr^{\rm BD}_\oeta$ by switching the factors gives a canonical isomorphism
\[\bar{i}^*j_{!*}(\calA\boxtimes \calB)\;\simeq\; \bar{i}^*j_{!*}(\calB\boxtimes \calA)\;\simeq\; \Psi^{\rm BD}(\calB\boxtimes \calA) \;\simeq\; \bar{\Psi}_\fraka(\calA)\star\calB. \]
Let us explain how this implies part a) (Corollary 2.8 in [\emph{loc. cit.}]): Every object in $\Sat_{\fraka,\os}$ is a direct summand of $\bar{\Psi}_\fraka(\calA)$ for some $\calA$ because $\Sat_{\fraka,\os}$ is semi-simple with simple objects $\IC_{t^{\bar{\mu}}}$ and these are direct summands of $\bar{\Psi}_\fraka(\IC_\mu)$ with multiplicity $1$. Since convolution with $\bar{\Psi}_\fraka(\calA)$ is bi-exact, it follows that $\Sat_{\fraka,\os}$ is stable under convolution. This implies a).\medskip\\ 
b) The tuple $(\Sat_{\fraka,\os},\star)$ has a unique structure of a neutral Tannakian category such that
\[\bar{\Psi}_\fraka\colon  (\Sat_{G,\oeta},\star)\;\longto\; (\Sat_{\fraka,\os},\star)\]
is a tensor functor compatible with the fiber functors $\om_\oeta\simeq \om_{\os}\circ \bar{\Psi}_\fraka$.\smallskip\\
\S 3 in [\emph{loc. cit.}] carries over literally: In Theorem-Definition 3.1 of [\emph{loc. cit}] one may replace $\bbA^1_k$ by any smooth curve $X$. This implies that $\bar{\Psi}_\fraka\colon (\Sat_{G,\oeta},\star)\to (\Sat_{\fraka,\os},\star)$ is a central functor (Proposition 3.2 of [\emph{loc. cit.}]): properties (1)-(4) of a central functor in [\emph{loc. cit.}] can be checked the same way because the monodromy of all nearby cycles involved is trivial as above (cf. part a)) and hence can be expressed via intermediate extensions. Now applying Lemma 3.3 of [\emph{loc. cit.}], it follows that $\om_{\os}$ is a fiber functor (Corollary 3.5 of [\emph{loc. cit.}]). Note that $\om_\oeta\simeq \om_{\os}\circ \bar{\Psi}_\fraka$ follows by proper base change. We deduce that $\Sat_{\fraka,\os}$ is a neutral Tannakian category. The uniqueness of the Tannakian structure follows from the uniqueness of the symmetric monoidal structure for $\om_\oeta$, cf. the discussion above \eqref{fiberfun1}. This proves b).\medskip\\ 
c) There is up to natural isomorphism a commutative diagram of functors of abelian tensor categories 
\[
\begin{tikzpicture}[baseline=(current  bounding  box.center)]
\matrix(a)[matrix of math nodes, 
row sep=1.5em, column sep=2em, 
text height=1.5ex, text depth=0.45ex] 
{(\Sat_{G,\oeta},\star)&(\Sat_{\fraka,\os},\star) \\ 
(\Rep_{\algQl}(\hat{G}),\otimes)&(\Rep_{\algQl}(\hat{G}^I),\otimes), \\}; 
\path[->](a-1-1) edge node[above] {$\bar{\Psi}_\fraka$} (a-1-2);
\path[->](a-2-1) edge node[above] {$\res$} (a-2-2);
\path[->](a-1-1) edge node[right] {$\om_\oeta$} (a-2-1);
\path[->](a-1-2) edge node[right] {$\om_\os$} (a-2-2);
\end{tikzpicture}
\]
where the vertical arrows are equivalences.\medskip\\
Let $H=\Aut^\star(\om_\os)$ be the affine $\algQl$-group scheme of tensor automorphisms defined by $(\Sat_{\fraka,\os},\om_{\os})$. Via the unramified Satake equivalence, the tensor functor $\bar{\Psi}_\fraka$ defines a morphism $H\to \hat{G}$ which identifies $H$ with a closed reductive subgroup of $\hat{G}$. Indeed, every object in $\Sat_{\fraka,\os}$ appears as a direct summand in the essential image of $\bar{\Psi}_\fraka$, and since $\Sat_{\fraka,\os}$ is semi-simple, $H$ is reductive. It remains to identify the subgroup $H\subset \hat{G}$. The inertia group $I$ acts on $\Gr_{G,\oeta}\to \Gr_{G,\breve{\eta}}$ induced from the action on $\oeta\to \breve{\eta}$ where $\breve{\eta}=\Spec(\nF)$. As in the Appendix of [\emph{loc. cit.}], this induces via $\Sat_{G,\oeta}\times I\to \Sat_{G,\oeta}$, $(\calA,\ga)\mapsto \ga^*\calA$, an action of $I$ on the Tannakian category $(\Sat_{G,\oeta}, \om_\oeta)$, and hence on $\Aut^\star(\om_\oeta)=\hat{G}$. Since the tensor functor $\bar{\Psi}_\fraka$ is invariant under this action, we get that $H\subset \hat{G}^I$ (cf. Lemma 4.5 in [\emph{loc. cit.}]), and we need to show that equality holds. 
Recall that $\hat{G}$ admits a canonical pinning $(\hat{G},\hat{B},\hat{T},\hat{X})$, cf. Remark \ref{pinrem}. The Galois action, and in particular the $I$-action preserves the pinning. So we can apply Proposition \ref{keylem} and Corollary \ref{highweight} below which shows that $\Rep_{\algQl}(\hat{G}^I)$ is semi-simple with simple objects the irreducible highest weight representations $V_{\bar{\mu}}$ for $\bar{\mu}\in X^*(\hat{T}^I)_+$. Note that $X^*(\hat{T}^I)_+=X_*(T)_{I,+}$ as partially ordered semigroups. We claim that $V_{\bar{\mu}}|_H$ is the irreducible representation $\om_\os(\IC_{t^{\bar{\mu}}})$ and hence $H=\hat{G}^I$ (the argument below Lemma 4.10 in [\emph{loc. cit.}]): Let $\mu\in X_*(T)_+$ be a lift of $\bar{\mu}$. Let $V_\mu$ be the irreducible representation of $\hat{G}$ of highest weight $\mu$. Then 
\[V_\mu|_{\hat{G}^I} = V_{\bar{\mu}} \oplus \bigoplus_{\bar{\la}< \bar{\mu}}V_{\bar{\la}}^{\oplus c_{\bar{\la},\bar{\mu}}}\]
for some $c_{\bar{\la},\bar{\mu}}\geq 0$. 
By the geometric Satake equivalence, we have $V_\mu=\om_\oeta(\IC_\mu)$ and 
\[V_\mu|_H =\om_\os\circ \bar{\Psi}_\fraka(\IC_\mu)= \om_\os(\IC_{t^{\bar{\mu}}})\oplus\bigoplus_{\bar{\la}<\bar{\mu}} \om_\os(\IC_{t^{\bar{\la}}})^{\oplus d_{\bar{\la},\bar{\mu}}}\]
for some $d_{\bar{\la},\bar{\mu}}\geq 0$ because $\IC_{t^{\bar{\mu}}}$ is a direct summand of multiplicity $1$ in $\bar{\Psi}_\fraka(\IC_{\mu})$. Consider 
\[V_{\bar{\mu}}|_H \oplus \bigoplus_{\bar{\la}< \bar{\mu}}V_{\bar{\la}}|_H^{\oplus c_{\bar{\la},\bar{\mu}}}=V_\mu|_H=\om_\os(\IC_{t^{\bar{\mu}}})\oplus\bigoplus_{\bar{\la}<\bar{\mu}} \om_\os(\IC_{t^{\bar{\la}}})^{\oplus d_{\bar{\la},\bar{\mu}}}\]
By induction on $\bar{\mu}$, we obtain that $V_{\bar{\la}}|_H=\om_\os(\IC_{t^{\bar{\la}}})$ and $c_{\bar{\la},\bar{\mu}}=d_{\bar{\la},\bar{\mu}}$ for all $\bar{\la}<\bar{\mu}$ and hence $V_{\bar{\mu}}|_H=\om_\os(\IC_{t^{\bar{\mu}}})$. This finishes the proof of part c) and Theorem C from the introduction. The uniqueness of the equivalence in Theorem C is a consequence of the Isomorphism Theorem in the theory of reductive groups. \medskip\\  
\emph{Galois descent:} Based on the geometric equivalence above, one shows that 
\[(\Sat_\fraka,\star)\simeq(\Rep_{\algQl}({\LG}_{\on{r}}),\otimes), \;\;\;\;\calA\mapsto \om_s(\calA),\] 
as in \cite[Appendix]{RZ}. In particular, Theorem \ref{Zhu1} i) holds, and part iii) follows from part ii). 

For ii), let $\calA\in \Sat_G$. We claim that $\Psi_\fraka(\calA)\in \Sat_\fraka$. Indeed, $\Psi_\fraka(\calA)$ is pure of weight $0$, cf. Proposition \ref{monodromy2}, and it is enough to show that $\Psi_\fraka(\IC_\mu)\in P_{L^+\calG}(\Fl_\fraka)$ is semi-simple for all $\mu\in X_*(T)$. By replacing $k$ by a finite extension, we may assume that every $L^+\calG$-orbit is defined over $k$. The $L^+\calG$-equivariance implies that there is a finite direct sum decomposition
\[\Psi_\fraka(\IC_\mu)\;\simeq\;\bigoplus_{w}\IC_w\otimes V_w, \] 
where $\IC_w$ is the intersection complex of the Schubert variety $Y_w\subset \Fl_\fraka$, $w\in W$, and $V_w$ is a local system on $\Spec(k)$. In fact, $V_w$ is constant because $\om_s(\Psi_\fraka(\IC_\mu))\simeq \om(\IC_\mu)$, cf. Corollary \ref{satcor}. This shows $\Psi_\fraka(\calA)\in \Sat_\fraka$.

It remains to show that $\Psi_\fraka\colon  \Sat_G\to \Sat_\fraka$ is a tensor functor, i.e. that the isomorphism $\Psi_\fraka(\calA\star\calB)\simeq \Psi_\fraka(\calA)\star\Psi_\fraka(\calB)$ is Galois equivariantly compatible with the commutativity constraint, and defines a morphism in $P_{L^+\calG}(\Fl_\fraka\times_s\eta)$. This follows from the fact that the Beilinson-Drinfeld Grassmannians are defined over the ground field, cf. \cite[\S 9.b]{PZ}. The uniqueness is clear. This finishes the proof of the theorem.
\end{proof}

\begin{appendix}

\section{A Levi Lemma}\label{levilemapp}
Let $F$ be a discretely valued complete field with valuation ring $\calO_F$ and perfect residue field $k$ of cohomological dimension $\leq 1$. Let $G$ be a connected reductive group over  $F$. Let $A\subset G$ be a maximal $F$-split torus, and write $\scrA=\scrA(G,A)$ for the corresponding apartment in the building. Let $\Omega\subset \scrA$ be a set whose projection onto the semi-simple part $\scrA^{\text{ss}}$ is bounded. Bruhat and Tits \cite{BT2} associate to $\Omega$ a smooth affine group scheme $\calG_{\Omega}$ over $\calO_F$ with geometrically connected fibers whose generic fiber is $G$, and whose $\mathcal{O}_F$-valued points fix $\Omega$ pointwise. In this appendix, we show that the construction of $\calG_\Omega$ is compatible with Levi subgroups as follows.

Let $L\subset G$ be a semi-standard Levi with respect to $A$, i.e. $L$ is the centralizer of a subtorus in $A$. Let $R(G,A)$ be the relative root system of $G$, and let $R(L,A)$ be the root subsystem of $L$. Then $R(L,A)\subset R(G,A)$ is a root subsystem, and $L$ is generated by the centralizer $Z_G(A)$ and the root groups $U_a$ with $a\in R(L,A)$. 

Let $\scrA(L,A)$ be the apartment of $L$ corresponding to $A$. There is an identification $\scrA(G,A)=\scrA(L,A)$ as homogeneous spaces compatible with the action of the Iwahori Weyl groups. We also consider $\Omega$ as a subset of $\mathscr{A}(L,A)$. Then its projection onto the semi-simple part is bounded, and we denote by $\calL_{\Omega}$ the corresponding group scheme.

\begin{lem}\label{levilem} 
The group scheme $\calL_\Omega$ is scheme theoretic closure of $L$ in $\calG_{\Omega}$. In particular,
\[\calL_{\Omega}(\calO_F)=\calG_{\Omega}(\calO_F)\cap L(F).\]
\end{lem}
\begin{proof}
We may assume that $F=\nF$. Then $T=Z_G(A)$ is a maximal torus because $G$ is quasi-split. We appeal to the construction of $\mathcal{L}_{\Omega}$ and $\mathcal{G}_{\Omega}$.\\ 
Let $\mathcal{T}$ be the connected  lft-N\'eron model of $T$ over $\mathcal{O}_F$, cf.  \cite[\S 10.1-10.3]{BLR}. For a every root $a\in R$, Bruhat and Tits define a smooth group scheme $\calU_{a,\Omega}$ with geometrically connected fibers whose generic fiber is the root subgroup $U_a$. The groups $\calU_{a,\Omega}$ are the same for $G$ as for $L$. There is a closed immersion
\[
\frakX_L\defined \calT\times\prod_{a\in R(L,A)}\calU_{a,\Omega}\longto\calT\times\prod_{a\in R(G,A)}\calU_{a,\Omega}\defined\frakX_G.
\]
Let $\frakX'_L$ (resp. $\frakX'_G$) be the scheme obtained by gluing $\frakX_L$ (resp. $\frakX_G$) and $L$ (resp. $G$) along $(\frakX_L)_\eta$ (resp. $(\frakX_G)_\eta$). We obtain a closed immersion $\frakX'_L\to\frakX'_G$ of smooth separated $\calO_K$-schemes compatible with the birational group laws. The scheme $\calG_{\Omega}$ (resp. $\calL_{\Omega}$) is constructed from $\frakX'_G$ (resp. $\frakX'_L$) by extending the birational group law. \\
Denote by $\bar{L}$ the scheme theoretic closure of $L$ in $\calG_{\Omega}$. The scheme theoretic closure $\bar{\frakX}'_L$ of $\frakX'_L$ in $\calG_{\Omega}$ is equal to $\bar{L}$, and the group law on $\bar{L}$ extends  the birational group law on $\frakX'_L$. \\
The scheme $\bar{L}$ is a flat closed (hence affine) $\calO_F$-subgroup scheme of $\calG_\Omega$. By the uniqueness of the extension of birational group laws, it is enough to show that $\bar{L}$ is smooth, i.e. that the special $\bar{L}\otimes k$ fiber is smooth. But $\bar{L}\otimes k$  contains the open smooth subscheme $\frakX'_L\otimes k$, hence is smooth.
\end{proof}
\begin{rmk}\label{levirem}
Let $\Omega=\fraka$ be a facet of $\scrA(G,A,F)$. Denote by $\fraka_L$ the unique facet of $\mathscr{A}(L,A,F)$ with $\fraka\subset\fraka_L$. Then $\calL_{\fraka_L}=\calL_\fraka$, and Lemma \ref{levilem} implies that the scheme theoretic closure of $L$ in $\calG_{\fraka}$ is the parahoric group scheme over $\calO_F$ associated with $\fraka_L$.
\end{rmk}

\section{Reductive descent}\label{reductivedescentapp}
This appendix is divided into two parts. Let $k$ be a henselian valued field with completion denoted $K$. In the first part, we show that every connected reductive $K$-group $G$ descends to a connected reductive $k$-group, cf. Corollary \ref{reductivedescent1}. The beautiful argument was explained to me by B. Conrad, and I thank him for the permission to give his proof. The key step (Theorem \ref{generalresult} below) is a special case of a result due to Gabber-Gille-Moret-Bailly \cite[Prop. 3.5.3 (2)]{GGMB}. The second part is an application of the first part. If $k$ is discretely valued, we prove that every flat affine model of $G$ over the valuation ring of $K$ descends to the valuation ring of $k$, cf. Corollary \ref{reductivedescent2}. This uses the descent Lemma of Beauville and Laszlo \cite{BL}.

In a first step, we need some general result on Galois cohomology, cf. \cite[\S 3.5]{GGMB}. Note that the Galois groups of $k$ and $K$ are naturally isomorphic, since $k$ is henselian, cf. \emph{[loc. cit. \S 3.5.1]}. The proof below was explained to me by B. Conrad, and we give it for the reader's convenience.

\begin{thm}[\cite{GGMB} Prop. 3.5.3 (2)] \label{generalresult} For any smooth affine $k$-group $H$, the natural map
\[{\rm H}^1(k,H)\longto {\rm H}^1(K,H_K)\]
is bijective.
\end{thm}

\begin{proof}[Proof due to B. Conrad] By Galois-twisting, injectivity reduces to triviality of the kernel. In other words, if $\calE$ is an $H$-torsor over $k$ which has a $K$-point, then it has a $k$-point. More generally, if $X$ is a smooth $k$-scheme then $X(k)$ is dense in $X(K)$ for the valuation topoloy. This is Zariski-local on $X$, so we can assume there is an etale map $f\colon X\to \bbA^n_k$. The open image $V=f(X)$ is dense open in $\bbA^n_k$, so $V(k)$ is dense in $V(K)$ due to density of $k$ in $K$. By the Zariski-local structure theorem for \'etale morphisms and the $K$-analytic inverse function theorem, for each $x\in X(K)$ and $v=f(x)\in V(K)$, every $v'$ sufficiently near $v$ admits $x'\in f^{-1}(v')$ near $x$ in $X(K)$. By openness of $X(K)\to V(K)$, for any $x\in X(K)$ and open $\Om \subset X(K)$ around $x$ we can find an open $U\subset V(K)$ around $v=f(x)$ such that every $u\in U$ is the image of a $K$-point in $\Om$. Consider such $u\in V(k)\cap U$ as exists by density of $V(k)$ in $V(K)$. The fiber scheme $f^{-1}(u)$ is finite \'etale over $k$, so the equivalence of Galois theories of $k$ and $K$ shows that every $K$-point in $f^{-1}(u)$ comes from a unique $k$-point of $f^{-1}(u)$. Hence, we can find a $k$-point in $\Om$. This completes the proof of injectivity.

For surjectivity, choose a closed $k$-subgroup inclusion $j\colon H\hookto \Gl_n=:G$ and let $X=G/H$ which is a smooth $k$-scheme. Thus, there is a natural surjection
\[G(k)\bslash X(k)\onto {\rm H}^1(k,H)\]
since ${\rm H}^1(k,\Gl_n)=1$ and likewise for $K$. It therefore suffices to show that the natural map
\[X(k)\to G(K)\bslash X(K)\]
is surjective. Since $X(k)$ is dense in $X(K)$ by the above, it suffices to show that all $G(K)$-orbits in $X(K)$ are open. But each orbit map $G_K\to X_K$ through a $K$-point is a smooth map since $H$ is smooth, so the induced map on $K$-points is open (hence has open image) by the $K$-analytic inverse function theorem (using the Zariski-local structure of smooth morphisms).
\end{proof}

\begin{cor}[B. Conrad] \label{reductivedescent1}
Let $G$ be a connected reductive $K$-group. Then $G$ descends to a connected reductive $k$-group, i.e. there exists a connected reductive $k$-group $G'$ unique up to isomorphism such that as groups $G'_K \;\simeq\; G$.
\end{cor}
\begin{proof}
The Galois groups of $k$ and $K$ are naturally isomorphic, so if $R$ denotes the root datum of $G_{\bar{K}}$ then ${\rm H}^1(k,\Aut(R))\to {\rm H}^1(K,\Aut(R))$ is bijective. This says that every quasi-split connected reductive $K$-group descends to a quasi-split connected reductive $k$-group which is moreover unique up to isomorphism (since these ${\rm H}^1$'s classify quasi-split forms with a given geometric root datum). Every connected reductive $K$-group $G$ has a unique quasi-split inner form $G^*$, so $G$ is obtained from $G^*$ via twisting against a class in ${\rm H}^1(K,G^*_{\ad})$ for the adjoint semisimple $G^*_{\ad}:=G^*/Z_{G^*}$. Thus, it suffices to show that ${\rm H}^1(k,H)\to {\rm H}^1(K,H_K)$ is surjective for every smooth affine $k$-group $H$. The corollary follows from Theorem \ref{generalresult}.
\end{proof}

Let us point out the following consequence which we use throughout the main body of the manuscript. Denote by $\calO_k$ (resp. $\calO_K$) the valuation ring of $k$ (resp. $K$).

\begin{cor} \label{reductivedescent2}
Assume $k$ to be discretely valued. Let $\calG$ be a flat affine $\calO_K$-group scheme with $\calG_K$ a connected reductive $K$-group. Then $\calG$ descends to a flat affine $\calO_k$-group, i.e. there exists a flat affine $\calO_k$-group scheme $\calG'$ unique up to isomorphism such that $\calG'_{\calO_K}\simeq \calG$ as groups. Moreover, if $\calG$ is of finite type over $\calO_K$, then $\calG'$ is of finite type over $\calO_k$.
\end{cor}
\begin{proof}
Let $\pi$ be a uniformizer in $k$. Then $k=\calO_k[{1\over \pi}]$ (resp. $K=\calO_K[{1\over \pi}]$) and $\calO_K$ is the $\pi$-adic completion of $\calO_k$. Denote $\calG=\Spec(\calA)$ where $\calA$ is a $\pi$-regular $\calO_K$-algebra by the flatness of $\calG$. Since $\calG_K$ is connected reductive, there is by Corollary \ref{reductivedescent1} a connected reductive $k$-group $G'=\Spec(A')$ and a $K$-isomorphism $u: \calG_K\simeq G'_K$, i.e. a $K$-isomorphism $u^\sharp: A'_K \simeq \calA_K$ of Hopf algebras. By the Theorem of Beauville and Laszlo \cite{BL} applied to the triple $(A',\calA,u^\sharp)$, there is a triple $(\calA',\al,\be)$ with $\calA'$ a $\pi$-regular $\calO_k$-module, $\al:\calA'_k \simeq A'$ and $\be: \calA'_{\calO_K}\simeq \calA$ unique up to unique isomorphism such that $u^\sharp= \be_K\circ \al_K^{-1}$. Now the uniqueness implies that $\calA'$ is a $\calO_k$-Hopf algebra. Let $\calG'\defined \Spec(\calA')$ an affine $\calO_k$-group scheme. Since $\calO_k$ is a discrete valuation ring, the $\pi$-regularity of $\calA'$ implies flatness of $\calG'$. Now assume $\calG\to \Spec(\calO_K)$ to be of finite type. Since $\calO_k\subset \calO_K$ is faithfully flat, fpqc-descent implies that $\calG'\to \Spec(\calO_k)$ is of finite type. This proves the corollary.
\end{proof}

\section{The group of fixed points}\label{fixpointapp}
Let $G$ be a connected reductive group over an algebraically closed field $C$. Let $I$ be a finite subgroup of the algebraic automorphisms of $G$, and assume that $I$ fixes some pinning of $G$. As is proven by Haines \cite[Proposition 4.1]{Haines14} relying on work of Steinberg \cite{St} and Springer \cite{Sp}, the group of fixed points $G^I$ is again a reductive group which is not connected in general. In this appendix, we prove the existence and uniqueness of irreducible highest weight representations of $G^I$. 

First recall the notion of a pinning. Let $T\subset B\subset G$ be a maximal torus contained in a Borel subgroup. Let $R=R(G,T)$ (resp. $R^\vee$) be the set of roots (resp. coroots), and let $R_+=R(B,T)$ (resp. $R_+^\vee$) be the subset of positive roots (resp. coroots). There is a bijection $R\to R^\vee$, $a\mapsto a^\vee$ which preserves the subsets of positive roots. For $a\in R$, let $U_a\subset G$ be the root subgroup, and denote by $\fraku_a\subset \Lie(G)$ its Lie algebra. Denote by $\Delta\subset R^+$ (resp. $\Delta^\vee\subset R_+^\vee$) the set of simple roots (resp. coroots). For every $a\in \Delta$, choose a generator $X_a$ of the $1$-dimensional $C$-vector space $\fraku_a$, and let $X=\sum_{a\in\Delta}X_a$ be the principal nilpotent element in $\Lie(B)$. A \emph{pinning of $G$} is a quadruple $(G,B,T,X)$ where $T\subset B$ is a torus contained in a Borel subgroup, and $X\in \Lie(B)$ is a principal nilpotent element. Note that there is a canonical isomorphism 
\begin{equation}\label{basedrootiso}
\Aut((G,B,T,X))\;\simeq\;\Aut((X^*(T),R,\Delta,X_*(T),R^\vee,\Delta^\vee))
\end{equation}
between the pinning preserving automorphisms of $G$, and the automorphisms of the based root datum $(X^*(T),R,\Delta,X_*(T),R^\vee,\Delta^\vee)$. 

For readability we mention the following basic facts on groups of fixed points. Let $H$ be any affine group scheme over $C$, and let $I\subset \Aut_C(H)$ be a finite subgroup of algebraic automorphisms. Then the group of fixed points $H^I\subset H$ is a closed subgroup scheme. Assume that the order $|I|$ is prime to the characteristic of $C$. Then\smallskip\\
a) if $H$ is smooth, then $H^I$ is smooth, cf. \cite[3.4]{Ed}, and\smallskip\\
b) if $H$ is reductive, then $H^{I}$ is reductive, cf. \cite[Theorem 2.1]{PY},\smallskip\\
but not connected in general. If one assumes further that $I$ fixes a pinning of $H$, one may even drop the `prime to the characteristic'-hypothesis on $|I|$ and still having a) and b) above. In \cite[Proposition 4.1]{Haines14}, Haines proves the following result.

\begin{prop}[Haines]\label{keylem}
Let $G$ be a connected reductive group over an algebraically closed field $C$. Let $(G,B,T,X)$ be a pinning of $G$, and let $I$ be a finite subgroup of the pinning preserving automorphisms. \smallskip\\
i) The group $G^I$ is reductive, and there exists some principal nilpotent $X^I\in\Lie(G^{I,0}) $ such that the tuple $(G^{I,0},B^{I,0},T^{I,0},X^I)$ is a pinning of the connected reductive group $G^{I,0}$. If $\cha(C)\not = 2$, then $X^I=X$. \smallskip\\
ii) The inclusion $T^I\subset G^I$ induces a bijection on connected components $\pi_0(T^I)\simeq \pi_0(G^I)$.
\end{prop}
\hfill\ensuremath{\Box} 

Let $Q_+\subset X^*(T)$ be the semigroup generated by $R_+$, and denote by $(Q_I)_+$ the image of $Q_+$ under the canonical projection $X^*(T)\to X^*(T^I)$. The group of characters $X^*(T^I)$ is equipped with the dominance order as follows. For $\mu, \la\in X^*(T^I)$, define $\la\leq\mu$ if and only if $\mu-\la\in (Q_I)_+$.   

Denote by $X^*(T)_+$ the semigroup of dominant weights, and let $X^*(T^I)_+$ be the semigroup defined as the image of $X^*(T)_+$ under the canonical projection $X^*(T)\to X^*(T^I)$. 
\begin{dfn}
Let $\mu \in X^*(T^I)$. An algebraic representation $\rho\colon G^I\to \GL(V)$ is said to be \emph{of highest weight $\mu$} if \smallskip\\
i) $\mu$ appears with a non-zero multiplicity in the restriction $\rho|_{T^I}$, and\smallskip\\
ii) if $\la\in X^*(T^I)$ appears in $\rho|_{T^I}$ with non-zero multiplicity, then $\la\leq\mu$.
\end{dfn}

\begin{rmk}
Let $w_0$ be the longest element in the finite Weyl group $W_0=W_0(G,T)$. Since $I$ acts by pinned automorphisms, we have $w_0\in W_0^I$, and it follows that $w_0$ acts on $X^*(T^I)$. Then property ii) implies that $w_0\mu\leq\la\leq \mu$, for all $\la\in X^*(T^I)$ appearing in $\rho|_{T^I}$ with non-zero multiplicity.  
\end{rmk}

If $G^I$ is connected reductive, then $T^I$ is a maximal torus by Proposition \ref{keylem}. In this case, it is well-known that there exists for every $\mu\in X^*(T^I)_+$ a unique up to isomorphism irreducible representation of highest weight $\mu$, and that every irreducible representation is of this form. Moreover, the multiplicity of the $\mu$-weight space is $1$, cf. \cite[Chapter II.2]{Jantzen}. 

\begin{cor} \label{highweight} Let $G$ be a connected reductive group over an algebraically closed field $C$. Let $(G,B,T,X)$ be a pinning of $G$, and let $I$ be a finite subgroup of the pinning preserving automorphisms. \smallskip\\
i) For every $\mu\in X^*(T^I)_+$ there exists a unique up to isomorphism irreducible representation $\rho_\mu$ of $G^I$ of highest weight $\mu$, and every irreducible representation of $G^I$ is of this form.\smallskip\\  
ii) The multiplicity of the $\mu$-weight space is $1$.
\end{cor}
\begin{proof} We follow the argument of Zhu \cite[Lemma 4.10]{RZ}. Let $\bar{\mu}$ be the image of $\mu$ under the restriction $X^*(T^I)\to X^*(T^{I,0})$, and let $\rho_{\bar{\mu}}$ be the unique irreducible representation of highest weight $\bar{\mu}$, cf. \cite[Chapter II.2]{Jantzen}. Frobenius reciprocity and Proposition \ref{keylem} imply 
\begin{equation}\label{frobrec}
\ind_{G^{I,0}}^{G^I}(\rho_{\bar{\mu}})\;\simeq\;\bigoplus_{\chi\in X^*(\pi_0(T^I))}\rho\otimes\chi,
\end{equation}
where $\rho$ is an irreducible representation of $G^I$ which restricts to $\rho_{\bar{\mu}}$. Here, the $\chi$'s are considered as $G^I$-representations by inflation along $G^I\to\pi_0(G^I)\simeq \pi_0(T^I)$. This shows that there is a unique $\chi\in X^*(\pi_0(T^I))$ such that $\rho_\mu=\rho\otimes\chi$ is of highest weight $\mu$. Conversely, \eqref{frobrec} implies that every irreducible representation of $G^I$ is a direct summand of some induction, and hence is of the form $\rho_\mu$ for some $\mu\in X^*(T^I)$. This proves i). Part ii) is easily deduced from \eqref{frobrec}. 
\end{proof}

\end{appendix}


\end{document}